\documentclass[a4paper,10pt]{article}
\usepackage{amsmath, amsthm, amsfonts, amssymb}
\usepackage{mathrsfs}
\usepackage[normalem]{ulem}

\usepackage{xcolor}

\newcommand{\be}{\begin{equation}}
\newcommand{\ee}{\end{equation}}

\newcommand{\ba}{\begin{equation}\begin{aligned}}
\newcommand{\ea}{\end{aligned}\end{equation}}

\newcommand{\ex}{\mathrm{e}}
\newcommand{\di}{\mathrm{d}}

\newcommand{\vf}{\varphi}
\newcommand{\ve}{\varepsilon}
\newcommand{\cF}{{\mathcal F}}

\newcommand{\wt}{\widetilde}

\newcommand{\mbN}{{\mathbb N}}
\newcommand{\mbZ}{{\mathbb Z}}
\newcommand{\mbR}{{\mathbb R}}

\newcommand{\1}{\mathbb{I}}
\renewcommand{\lg}{\langle}
\newcommand{\rg}{\rangle}
\newcommand{\Var}{{\mathop{\rm Var}}}
\newcommand{\Pb}{{\mathop{\rm P}}}
\renewcommand{\P}{{\mathop{\rm P}}}

\newcommand{\cX}{\mathcal{X}}

\newcommand{\cN}{\mathcal{N}}
\newcommand{\cZ}{\mathcal{Z}}

\newcommand{\E}{{\mathop{\rm E}}}

\newcommand{\sgn}{\mathop{\rm sgn}}

\newcommand{\toP}{\overset{\Pb}{\to}}

\theoremstyle{plain}
\newtheorem{thm}{Theorem}[section]
\newtheorem{lem}[thm]{Lemma}

\newtheorem{corl}[thm]{Corollary}

\theoremstyle{definition}
\newtheorem{remk}[thm]{Remark}
\newtheorem{expl}[thm]{Example}

\numberwithin{equation}{section}

\let\oldmarginpar\marginpar
\renewcommand{\marginpar}[1]{\oldmarginpar{\scriptsize\texttt{\color{red}{#1}}}}

\title{Walsh's Brownian Motion\\ and Donsker Scaling Limits\\ of Perturbed Random Walks
}

\author{Ilya Pavlyukevich\footnote{Institute of Mathematics, Friedrich Schiller University Jena, Ernst--Abbe--Platz 2, 
07743 Jena, Germany; ilya.pavlyukevich@uni-jena.de}
 \quad    and \quad Andrey Pilipenko\footnote{Institute of Mathematics, National Academy of Sciences of Ukraine, Tereshchenkivska Str.\ 3, 01601, Kyiv, 
Ukraine} \footnote{National Technical University of Ukraine 
``Igor Sikorsky Kyiv Polytechnic Institute'',
ave.\ Pobedy 37, Kyiv 03056, Ukraine; pilipenko.ay@gmail.com} 
    }

\begin{document}

\maketitle

\begin{abstract}
In this paper we study Markov chains with the state space given by the coordinate axes of $\mathbb R^m$,
$m \geq 2$,
whose step sizes on each positive half-axis are distributed according to a centered probability distribution
with variance $v_i^2 \in (0, \infty)$, $i = 1,\ldots, m$.
Under very mild assumptions on the jumps sizes on the negative
half-axes, we show that the Donsker scaling limit of such Markov chains is a Walsh Brownian motion
whose weights are determined explicitly in terms of stationary distributions of certain embedded Markov
chains. This convergence result is applied to integer-valued random walks perturbed on a finite subset of
$\mathbb Z$ called a \emph{membrane}. We show that their Donsker scaling limit is an oscillating skew Brownian motion.
\end{abstract}



\tableofcontents
\smallskip

\section{Introduction}

Let $S_\xi(n):=\xi_1+\dots +\xi_n$, $n\geq 0$, be a random walk where $(\xi_k)$ are independent copies of an integer-valued random variable $\xi$. 
The well known Donsker invariance principle (see, e.g., Theorem 14.1 in Billingsley \cite{billingsley2013convergence})
states that a properly scaled random walk converges in law to the standard Brownian motion if $\E \xi=0$ and 
$v^2:= \operatorname{Var} \xi\in (0,\infty)$. Namely,
\ba
\frac{S_\xi([n\cdot])}{v \sqrt n} \Rightarrow W(\cdot), \ n\to\infty,
\ea
in the Skorokhod space $D=D([0,\infty),\mathbb R)$.

Consider an integer-valued Markov chain $X$ whose transition probabilities coincide with transition probabilities of $S_\xi$ everywhere except for a finite set $A\subset\mbZ$. We call the set $A$ a \emph{membrane} and say that $X$ is a random walk perturbed on $A$. 
It appears  that the Donsker scaling limit of $X$ is not a Brownian motion any longer but a diffusion with a singular drift. 

The first result on this topic belongs to Harrison and Shepp \cite{HShepp-81}, who considered 
a symmetric Bernoulli random walk perturbed on a one-point membrane $A=\{0\}$. They proved that if
\ba
&\Pb(X(1)=x\pm 1\, |\, X(0)=x)=\frac12,\quad x\in\mbZ\backslash\{0\},\\
&\Pb(X(1)=  1\, |\, X(0)=0)= p\in[0,1], \\
&\Pb(X(1)=  -1\, |\, X(0)=0)=1-p,
\ea
then the processes $\{\frac{X([n\cdot])}{\sqrt{n}}\}_{n\geq 1}$ converge in distribution to a \emph{skew Brownian motion} (SBM) $W^\text{skew}_\gamma$ with permeability parameter
$\gamma=2p-1$. Recall that the SBM $W^\text{skew}_\gamma$ with permeability parameter $\gamma\in[-1,1]$ is  a continuous homogeneous  Markov process with transition probability density function 
\ba
p_t(x,y)= \varphi_t (x-y) + \gamma\sgn(y)\varphi_t (|x|+|y|),\quad x,y\in\mathbb{R},\quad t>0,
\ea
where $\vf_t(x)=(2\pi t)^{-1/2}\ex^{-\frac{x^2}{2t}}$, $x\in\mathbb R$, $t>0$. Another construction of the SBM was suggested by Itô and
McKean, see \cite[Problem 1, Section 4.2]{ItoMcKean-65} and Walsh \cite{walsh1978diffusion}. There, the SBM is obtained from a standard
reflecting Brownian motion by flipping its excursions independently with probability $1- p$, and leaving it
positive with probability $p$. An extensive survey on properties of a SBM can be found in Lejay \cite{Lejay-06}.

The approach by Harrison and Shepp \cite{HShepp-81} is classical.
To show convergence, they verified tightness and convergence of finite-dimensional distributions of the family
$\{\frac{X([n\cdot])}{\sqrt{n}}\}_{n\geq 1}$,
which can be done directly utilizing the simple structure of transition probabilities and Andr\'e's reflection principle.

The case of a general finite membrane $A =\{-d,\ldots , d\}$, $d \geq 0$, and arbitrary non-unit jumps outside of $A$
is much harder to deal with.
Some partial results on functional limit theorems that generalize the result by Harrison and Shepp \cite{HShepp-81} were obtained by 
Minlos and Zhizhina \cite{minlos1997limit}, Pilipenko et al.\
\cite{pilipenko2012limit,IksanovPilipenko2016,pilipenko2015limit,pilipenko2015impurity},
Ngo and Peign\'e \cite{NgoPeigne19}. In these works, the convergence to a SBM was established
under various restrictive assumptions, such as boundedness of jumps outside of the membrane, one-point membrane structure, etc. 

In this paper we obtain a very general convergence results (Theorems \ref{thm:Walsh_lim_pertRW}
and \ref{thm:Skew_lim_pertRW}) that use a minimal
set of technical assumptions and generalize all the previous work. Our argument relies on consideration of
a SBM as a particular case of a two-ray Walsh Brownian motion (WBM).

Recall that an $m$-ray WBM with non-negative weights $\mathbf p = (p_1,\ldots , p_m)$, $p_1 +\cdots + p_m = 1$, can be
constructed similarly to the approach by It\^o and McKean. Consider $m$ rays with a common endpoint $0$. A
WBM is a time-homogeneous continuous strong Markov process that spends zero time at $0$ and behaves like
a standard Brownian motion on each ray up to hitting $0$. Its excursions ``select'' the $k$-th ray with probability
$p_k$, see Walsh \cite{walsh1978diffusion} and Section 2 below. It is clear, that a 2-ray WBM can be naturally identified with a
SBM with $\gamma = p_1 - p_2$.

Starting with a perturbed Markov chain $X$ on $\mathbb Z$ with a membrane $A = \{-d,\ldots , d\}$, we will construct an
auxiliary Markov chain $\cX$ on the enlarged state space $\mathbb Z\times\{-, +\}$, such that the ``layers''
$\mathbb Z\times\{-\}$ and $\mathbb Z\times\{+\}$
will be in some sense identified with the negative and positive rays $\{\ldots ,-d - 2, -d - 1\}$ and $\{d + 1, d + 2,\ldots\}$
located to the left and to the right of the membrane of the Markov chain $X$. The transition probabilities of
$\cX$ on the positive layers $\mathbb N \times \{-\}$ and $\mathbb N \times \{+\}$ will coincide with those of $X$ to the left and to the right of the
membrane. The dynamics of $X$ within the membrane as well as jumps over the membrane will be mimicked
by (short) visits of $\cX$ to the negative layers $\{\ldots , -2, -1, 0\}\times \{-\}$ and $\{\ldots ,-2,-1, 0\}\times \{+\}$. Employing
the martingale characterization technique, we establish convergence of the Donsker scalings of the Markov
chain $\cX$ to a WBM, and thus, convergence of a properly rescaled Markov chain $X$ to a SBM.

This method of unfolding of a perturbed one-dimensional Markov chain into a multidimensional Markov
chain was initially suggested by Iksanov and Pilipenko in \cite{IksanovPilipenko2016}. In Bogdanskii et al.\ \cite{BPP22},
this method was
used to study convergence of $m$-dimensional random walks perturbed on a two-sided periodic hyperplane
membrane. In particular, in our proofs we will utilize a martingale characterization of the WBM obtained in
\cite{BPP22}, see Section 2. We emphasize that our approach allows to treat random walks with different variances on
different sides of the membrane. This case has not been treated in previous works with the only exception
of the paper \cite{NgoPeigne19}. The general result on convergence of ``unfolded'' Markov chains to a $m$-ray WBM is our
Theorem~\ref{thm:Walsh_lim_pertRW}.

For reader's convenience, we formulate here an easy-to-use version of the assumptions a convergence
result for scaling limits of perturbed random walks on $\mathbb Z$.

Let $(X(k))_{k\geq 0}$ be an integer-valued time-homogeneous Markov chain with a membrane $\{-d, \dots,d\}$, $d\geq 0$, such that  
\ba
&\P(X(1)=x+y|X(0)=x)=\P(\xi_+=y),\quad x>d,\\
&\P(X(1)=x+y|X(0)=x)=\P(\xi_-=y),\quad x<-d,
\ea
where the integer valued random variables $\xi_-$, $\xi_+$ satisfy  
\ba
\E\xi_\pm=0,\quad v_\pm^2=\operatorname{Var}\xi_\pm\in (0,+\infty),
\ea
and each of random variables $\xi_+$ and $\xi_-$ generates a 1-arithmetic random walk.

Assume also that  the states $\mbZ\backslash\{-d,\dots,d\}$ of the Markov chain $X$ communicate, 
$X$ exits from the membrane with probability one, and jumps from the membrane are integrable, i.e., 
\ba
\max_{|x|\leq d} \E\big[|X(1)|\, \big|\,  X(0)=x\big]<\infty.
\ea
Let $\vf\colon \mbR\to\mbR$ be the following piece-wise linear function: 
\ba
\label{e:phi}
\vf(x):=x\cdot v^{-1}_{\sgn(x)}:=x\big( v_+^{-1}{\1_{x\geq 0}} +v_-^{-1}{\1_{x< 0}}\big).
\ea
Then for any initial distribution of $X(0)$,  the Donsker scaling
$\big\{\vf\big(\frac{X([n\cdot])}{\sqrt{n}}\big)\big\}_{n\geq 1}$ weakly converges in $D([0,\infty),\mbR)$
to a SBM
 $W^\text{skew}_\gamma(\cdot)$
 starting at 0 with some permeability parameter $\gamma\in(-1,1)$.
 The rigorous  statement of the result as well as the explicit value of the permeability parameter $\gamma$
will follow from Theorem \ref{thm:Skew_lim_pertRW} and Remark~\ref{r:3.7}. Its formulation requires a consideration
of additional construction: entrance and exit embedded Markov chains, 
existence of their stationary distributions, etc. 
However, formula \eqref{eq:gamma_probab} for the permeability parameter $\gamma$ is short and natural.

This result allows us to show that the limit of the Donsker scaling
$\big\{\frac{X([n\cdot])}{\sqrt{n}} \big\}_{n\geq 1}$ is a oscillating SBM
(see Corollary \ref{c:sobm}), which is a natural generalization of the oscillating Brownian motion
originally introduced by Keilson and Wellner in \cite{keilson1978oscillating}. The oscillating SBM
is a diffusion that behaves like a Brownian motion with variance $v_+^2$ on the positive half-line and like 
a Brownian motion with variance $v_-^2$ on the negative half-line, and exhibits an infinitesimal push at $0$ in the
negative or positive direction.
We mention here the work \cite{hairer2010one} by Hairer and Manson who obtained the process $Y$
as a limit in the periodic homogenization problem with an interface. 

We will also prove a result on convergence of Donsker scalings of a perturbed random walk on a graph
$\mathbb N^m \cup \{0\}$.
Under natural assumptions on integrability of jumps the limit process will be a WBM, too, see
Theorem~\ref{thm:Walsh_non_rigorous}.
The result of Theorem~\ref{thm:Walsh_non_rigorous} extends the result of Enriques and Kifer \cite{EnriquesKifer}.
There the authors studied the Donsker scaling of symmetric random walks
with the step size $\ve>0$ on a spider graph consisting
of $m$ rays with the common origin $0$. They obtained convergence to a WBM with equal weights $1/m$ for
each ray. If the Markov chain lives on a $\ve$-lattice containing the origin, the results of
\cite{EnriquesKifer}
are covered by
our Theorem~\ref{thm:Walsh_non_rigorous}. For other initial values, an intermediate approximation argument has to be applied, see
Examples \ref{ex:4.3} and \ref{ex:4.4} for more detail. Note that our Theorem~\ref{thm:Walsh_non_rigorous}
is more general since its allows non-uniform
weights of the limit WBM and general step sizes of the Markov chain with different variances on each ray.

Eventually we mention that the case when jumps from the membrane are not integrable and belong to a domain of attraction of an
$\alpha$-stable law with $\alpha\in (0,1)$ are much harder to deal with even if jumps from the membrane are positive. 
Some partial results in this direction were obtained 
by Pilipenko with co-authors in \cite{pilipenko2020limit,iksanov2023functional,pilipenko2023boundary}, 
where a reflected Brownian motion with jump-type exit from 0 
appeared as a limit process. 
The case when jumps outside of the membrane are not in Donsker's setup but belong to a domain of attraction of 
a stable law with parameter in $(1,2)$ leads to a skew stable L\'evy process as it was recently shown in \cite{iksanov2023skew,dong2023discrete}.

\medskip

The structure of the paper is as follows. In Section~\ref{section:Walsh_and_skew}
we recall the definition of the WBM together with its
martingale characterization and its relation to the skew Brownian motion. We formulate main general results
on limit behavior of perturbed random walks in Section~\ref{sec:Main}.
The most important result is Theorem~\ref{thm:Walsh_lim_pertRW}, where
we consider perturbed random walks on $\mathbb Z\times \{1,\ldots,m\}$,
whose Donsker scaling limit is a WBM. The result for
scaling limits of a perturbed random walk with values in $\mathbb Z$
(Theorem~\ref{thm:Skew_lim_pertRW}) and the result for scaling limits
of a random walk on a graph (Theorem~\ref{thm:Walsh_non_rigorous})  follow from Theorem~\ref{thm:Walsh_lim_pertRW}
and the ideas of its proof. We determine the weights of the WBM and the
permeability parameter of the skew Brownian motion explicitly in terms of stationary distributions of some
embedded entrance and exit Markov chains. We also write down a stochastic differential equation for a skew
oscillating Brownian motion, which is a limit of scalings of perturbed random walk (Corollary 3.8).
Section~\ref{s:exa}
contains six illustrative examples. In particular, we show that formulas obtained in Theorems 3.3 and 3.5
include and extend several known results as special cases. The proof of Theorem 3.3 is presented in Section 5. Theorem 3.5
is proven in Section 6. The last Section 7 contains the proofs of Corollary 3.8 and Theorem 3.9.

\medskip

\noindent
\textbf{Acknowledgments:}
A.P.\ acknowledges support by the National Research Foundation
of Ukraine (project 2020.02/0014 ``Asymptotic regimes of perturbed random walks: on the edge of modern and classical
probability'') and thanks the Institute of
Mathematics of the FSU Jena for hospitality.
The authors are grateful
to the anonymous referee for their valuable comments that significantly improved the paper.


\section{Walsh's and Skew Brownian motions}\label{section:Walsh_and_skew}

The main result of this paper will be expressed in terms of Walsh's 
Brownian motion $W_\mathbf{p}$.
In this section we recall its definition, properties and martingale characterizations.

Originally, Walsh's Brownian motion was introduced in the
work \cite{walsh1978diffusion} as a diffusion on $m$ rays on a two-dimensional plane
with a common origin.
Away from the origin, on each ray, WBM is a standard one-dimensional Brownian motion that
can change the ray upon hitting the origin. In other words, 
the rays are characterized
by $m$ non-negative weights $\mathbf{p}=(p_1,\dots,p_m)$, $p_1+\cdots+p_m=1$. The number $p_i$ is the 
probability to leave the origin into the ray number $i$.

In our work \cite{BPP22}, we proposed a realization of the WBM as an $m$-dimensional process (as was also indicated in Walsh \cite{walsh1978diffusion})
living on positive semi-axes. In comparison to the 2-dimensional construction
with the help of polar coordinates, in this realization one does not have to 
treat the origin as a particular point at which the ray number is not well defined.

Let $m\geq 1$, 
\begin{equation}
\label{e:Em}
E^m=\{x\in\mathbb R^m\colon x_i\geq 0\text{ and } x_ix_j=0,\ i\neq j,\  i,j=1,\dots, m\},
\end{equation}
and $\mathbf{p}=(p_1,\dots,p_m)$ be a vector of weights.
Let us work on the canonical probability space of continuous $\mathbb R^m$-valued functions equipped with the filtration 
$\mathbb F$
generated by the coordinate 
mappings.
We say that 
WBM is a time-homogeneous continuous 
Feller process $W_\mathbf{p}=X=(X_1,\dots,X_m)$ on $E^m$
whose marginal distributions are given by
\begin{equation}
\begin{aligned}
\E_0 \text{e}^{\lambda_1 X_1(t)+\cdots +\lambda_m X_m(t) }&=\sum_{k=1}^m p_k \E_0 \text{e}^{\lambda_k |W(t)|},\\
\E_x \text{e}^{\lambda_1 X_1(t)+\cdots+\lambda_m X_m(t)}&=
\displaystyle 
\E_{x_j}\Big[   \1_{t<\tau_0}\text{e}^{\lambda_j W(t)}\Big]
+ \sum_{k=1}^m  p_k \E_{x_j}\Big[      \1_{t\geq\tau_0}\text{e}^{\lambda_k |W(t)|} \Big]\\
&= \E_{x_j}\Big[\sum_{k=1}^m p_k \text{e}^{\lambda_k |W(t)|}\Big]
 + \E_{x_j}\Big[ \1_{t<\tau_0} \Big( \text{e}^{\lambda_j W(t)}- 
\sum_{k=1}^m p_k\text{e}^{\lambda_k W(t)}\Big)\Big],\\
x&=(0,\dots,x_j,\dots,0),\ x_j>0,\ \lambda\in\mathbb C^m,
\end{aligned}
\end{equation}
where $W$ is a one-dimensional standard Brownian motion and $\tau_0=\inf\{t\geq 0\colon W_t=0\}$.
Note that the last expectation is the expectation of a standard Brownian motion killed upon hitting zero.

In Barlow et al.\ \cite{barlow1989walsh}, the authors provided 
the martingale characterization of the WBM realized as a process on the plane. In terms of 
the $m$-dimensional realization $X$ of the WBM, their characterization takes the following form.

\begin{thm}[Propositon 3.1 and Theorem 3.2 in Barlow et al.\ \cite{barlow1989walsh}]
\label{Walsh1}
Let $X=(X_1(t),\dots,X_m(t))_{t\geq 0}$ be a  
 continuous process. Then $X$ is a WBM
with parameters $p_1,\dots,p_m>0$ if and only if it satisfies 
the following conditions:
\begin{enumerate}
\item $X_i(t)\geq 0$ and $X_i(t)X_j(t)=0$ for all $i\neq j$ and $t\geq 0$;

\item for each $i=1,\dots,m$ the process
\begin{equation}
N_i(t):=(1-p_i) X_i(t)-p_i\sum_{j\neq i}X_j(t) ,\quad t\geq 0,
\end{equation}
is a continuous martingale;  
\item 
for each $i=1,\dots,m$ the process 
\begin{equation}
\label{e:brN}
(N_i(t))^2 -\int_0^t \Big((1-p_i) \1_{X_i(s)>0}-p_i  \1_{X_i(s)= 0}\Big)^2 \,\di s ,\quad t\geq 0,
\end{equation}
is a continuous martingale.  
\end{enumerate}
\end{thm}

Note that since the product of the indicator functions in \eqref{e:brN} is identically zero, we have
\begin{equation}
\begin{aligned}
\langle N_i\rangle_t=
\int_0^t \Big((1-p_i) \1_{X_i(s)>0}-p_i  \1_{X_i(s)= 0}\Big)^2 \,\di s
=\int_0^t \Big( (1-p_i)^2 \1_{X_i(s)>0}+p_i^2 \1_{X_i(s)= 0}\Big)\,\di s.  
\end{aligned}    
\end{equation}
Furthermore, it is clear, see Lemma 2.2 in Barlow et al.\ \cite{barlow1989walsh}, that the radial process
\begin{equation}
R(t):=\sum_{i=1}^m X_i(t)=\max_{1\leq i\leq m} X_i(t)
\end{equation}
is a reflecting Brownian motion, and hence it has a symmetric local time at $0$ defined by:
\begin{equation}
\label{e:loctime}
L^X_0(t):= L^R_0(t):=\lim_{\ve\to 0+}\frac{1}{2\ve}\int_0^t \1_{\max_{1\leq i\leq m} X_i(s) \leq \ve} \,\di s.
\end{equation}
In our paper \cite{BPP22} we gave an equivalent martingale characterization of the WBM that is more convenient for the proof of the convergence theorem of this paper.

\begin{thm}[Theorem 3.2 in \cite{BPP22}]
\label{thm:Walsh}
Let $X=(X_1(t),\dots,X_m(t))_{t\geq 0}$ and $\nu=(\nu(t))_{t\geq 0}$ be 
 continuous processes.
Then $X$ is a WBM with parameters $p_1,\dots,p_m>0$, 
 and $\nu$ is the local time of $X$ at $0$
 if and only if they satisfy   
the following conditions:
\begin{enumerate}
\item $X_i(t)\geq 0$ and $X_i(t)X_j(t)=0$ for all $i\neq j$, $t\geq 0$;

\item $\nu(0)=0$, $\nu$ is non-decreasing a.s., $\int_0^\infty \1_{X(s)\neq 0}\, \di \nu(s)=0$ a.s.;

\item the processes $M_1,\dots, M_m$ defined by
\begin{equation}
\label{e:Mnu}
M_i(t):=X_i(t)-p_i \nu(t),\quad t\geq 0,
\end{equation}
are continuous square integrable martingales 
with the predictable quadratic variations
\begin{equation}
\label{e:brM}
\langle M_i\rangle_t=\int_0^t \1_{X_i(s)>0} \, \di s.
\end{equation}
\item $\int_0^\infty \1_{X(s)=0} \,\di s=0$ a.s.
\end{enumerate}
\end{thm}

In his paper \cite{walsh1978diffusion}, Walsh addressed the properties of the 
skew Brownian motion originally constructed by It\^o and McKean in
\cite{ItoMcKean-65}, Problem 1, Section 4.2. A SBM with parameter $\gamma \in[-1,1]$
is a diffusion that behaves like a standard Brownian motion on the half-lines 
$(-\infty,0)$ and $(0,\infty)$ and upon hitting zero continues to the positive half-line with 
probability $p_+=(\gamma+1)/2$ and to the negative half line with probability  $p_-=(\gamma+1)/2$.  
In \cite{BPP22}, it was shown that a SBM can be obtained from a WBM by summing together the motions on different rays.

\begin{corl}[\cite{BPP22}, Corollary 3.4]
\label{thm:Walsh-skew}
Let $X=(X_1(t),\dots,X_m(t))_{t\geq 0}$ be a WBM with parameters $p_1,\dots,p_m>0$ and let $I\subseteq\{1,\dots, m\}$. Set
\begin{equation}
\gamma:=\sum_{i\in I} p_i - \sum_{j\in I^c} p_j= 
2\sum_{i\in I} p_i-1.
\end{equation}
Then $\gamma\in [-1,1]$ and the process
\begin{equation}
W_\gamma^{\emph{\text{skew}}}(t):=\sum_{i\in I} X_i(t)- \sum_{j\in I^c} X_j(t),\quad t\geq 0,
\end{equation}
is a skew Brownian motion with the parameter $\gamma$.
\end{corl}

Recall  that due to \cite{HShepp-81}, the skew Brownian motion is the unique strong solution of the stochastic differential equation
\begin{equation}
\label{e:SDEskew}
\di W_\gamma^\text{skew}(t)=\di W(t) + \gamma\,\di  L_0^{W^\text{skew}_\gamma}(t),
\end{equation}
where $L_0^{W^\text{skew}_\gamma}(\cdot)$ is the symmetric semimartingale local time at 0. Since  
$\langle W_\gamma^\text{skew} \rangle_t=t$, the symmetric semimartingale local time coincides with the symmetric local time.

In the sequel, we will use notation $W_\mathbf{p}(\cdot,x)$ and 
$W_\gamma^{{\text{skew}}}(\cdot,x)$ to denote
WBM and SBM starting at a point $x$.

\section{Convergence to WBM and SMB \label{sec:Main}}

One of the  main results of this paper is related to
a Markov chain $\cX=\big((R(k),l(k))\big)_{k\geq 0}$, on the state space $\cZ^m:=\mbZ\times \{1,\dots,m\}$, $m\geq 1$, that 
is defined by the following set of assumptions.

\medskip

\noindent
\textbf{A}$_1$. For all  $(x,i)\in  \cN^m:=\mbN\times\{1,\dots,m\}$
\ba
\Pb\Big(\cX(1)=(x+y,i) \Big|  \cX(0)=(x,i)\Big)=\Pb( \xi^{(i)}=y),
\ea
where integer-valued random variables $\xi^{(1)},\dots,\xi^{(m)}$ satisfy 
\ba
\E \xi^{(i)}=0,\quad  
v_i^2:=\Var\, \xi^{(i)}\in (0,\infty),\quad  1\leq i\leq m,
\ea 
and each  $\xi^{(i)}$, $i=1,\dots,m$, generates a 1-arithmetic random walk,
i.e., $\mathbb Z$ is the minimal lattice containing
$\operatorname{supp}(\xi^{(i)})$.

 \noindent
\textbf{A}$_2$.  For all $x \leq 0$ and $1\leq i\leq m$
\ba
\Pb(R(1)\in \mbN\  | \  \cX(0)=(x,i))=1.
\ea

\noindent
\textbf{A}$_3$. There is $C>0$ such that for all $x \leq 0$ and $1\leq i\leq m$
\ba
\E\Big[  R(1)  \Big|  \cX(0)=(x,i)\Big]\leq C(1+|x|).
\ea
\noindent
\textbf{A}$_4$.  The states $\cN^m$ of the Markov chain $\cX$ communicate. 

\medskip

We will  call $R(\cdot)$
the radius and $l(\cdot)$ the label of $\cX$. Note that the radius $R$ may attain non-positive values. The qualitative behaviour of 
$\cX$ is as follows.

Assumption \textbf{A}$_1$ says that for the ``positive'' initial point $\cX(0)=(R(0),l(0))=(x,i)$, $x>0$, the
radius $R$
behaves like a zero mean square integrable random walk with the increments having distribution
$\xi^{(i)}$ until
it hits $\{\ldots , -2, -1, 0\}$.
Therefore the first exit time
\ba
\label{e:ss}
\sigma&=\inf\{k\geq 0\colon R(k)\leq 0\}
\ea
is finite a.s.
Note that the label $l$ does not change up to the stopping time $\sigma$,
\ba
l(0)=l(1)=\cdots=l(\sigma)=i.
\ea
Upon reaching a non-positive value at time $\sigma$, the Markov chain $\cX$ 
for sure jumps to a ``positive'' value at one step. The label $l$ can change, too:
\ba
R(\sigma+1)>0, \quad l(\sigma+1)\in\{1,\dots,m\}.
\ea
The technical Assumption \textbf{A}$_3$ controls the size of the jump of the radius at time $\sigma+1$.

Let $\{\sigma_\iota\}_{\iota\geq 1}$ denote the sequence of exit times of $R$ from $\mbN$, 
that is  
\ba
\label{e:sigma}
\sigma=\sigma_1&=\inf\{k\geq 0\colon R(k)\leq 0\},\\
\sigma_{\iota+1}&=\inf\{k>\sigma_\iota\colon  R(k)\leq 0\}.
\ea
Let also
\be
\label{e:tau}
\tau_\iota:=\sigma_\iota+1,\quad \iota\geq 1,
\ee
be the entrance times to $\mbN$.

With the sequences $\{\sigma_\iota\}_{\iota\geq 1}$ and $\{\tau_\iota\}_{\iota\geq 1}$
we associate embedded \emph{exit} and \emph{entrance} Markov chains taking values in 
$\{\dots, -2,-1,0\}\times \{1,\dots,m\}$ and $\cN^m$ respectively and defined as
\ba
\label{e:EE}
\cX_\text{exit}(\iota)&:= \cX(\sigma_\iota), \\
\cX_\text{entrance}(\iota)&:= \cX(\tau_\iota), \quad \iota\geq 1,
\ea
Assumption {\bf A}$_4$ yields the following statement.

\begin{lem}
\label{lem:exit_entrance_stationary}
The Markov chains $\cX_\emph{exit} $ and $\cX_\emph{entrance}$ have unique stationary distributions $\pi_\emph{exit}$ and $\pi_\emph{entrance}$
respectively.
\end{lem}
\begin{proof}
1. First we prove the statement for the Markov chain $\cX_\text{exit}$.
It is well known that a Markov chain with
countably many states has a unique stationary distribution if and only if the set of states contains precisely
one positive recurrent class of essential communicating states, see Theorem 2 in Section 6 of Chapter 8 in
Shiryaev \cite{Shiryaev-2-19}.

It follows from Assumption $\mathbf{A}_4$ that all essential states of $\cX_\text{exit}$
communicate. All of them are either null
recurrent or positive recurrent. To exclude the case of null recurrence, it is sufficient to show that
there is $c>0$ such that for all $(x,i)\in \{\dots,-2, -1,0\}\times \{1,\dots, m\}$
\be
\label{eq:reccurence1}
\Pb\Big(\cX_\text{exit}(\iota+1)\in \{ 0\}\times\{1,\dots, m\}\,  \Big|\,  \cX_\text{exit}(\iota)=(x,i)\Big)\geq c.
\ee
To demonstrate this, it is sufficient to verify that there is $c>0$ such that for all $(x,i)\in \cN^m$
\be
\label{eq:reccurence469}
\Pb\Big(\cX(\sigma)=(0,i) \, \Big| \, \cX(0)=(x,i)\Big)\geq c.
\ee
We borrow the reasoning from Section 2 in Vysotsky \cite{vysotsky2015limit} and use facts
from the renewal theory of random walks.

Let $i$ be fixed. Consider a random walk
\ba
S_{-\xi^{(i)}}(0)&=0,\\
S_{-\xi^{(i)}}(k)&= -\xi^{(i)}_1 -\cdots -\xi^{(i)}_k,\quad k\in\mathbb N,
\ea
and let for $z\in\mathbb Z$
\ba
\chi_z=\inf\{k\geq 0\colon S_{-\xi^{(i)}}(k)>z\}.
\ea
Then for all $x\in\mathbb N$ and
  $y\in \mbN\cup\{0\}$
\ba
\label{e:B}
 \Pb\Big(\cX(\sigma)=(-y,i) \Big| \cX(0)=(x,i)\Big)=
 \Pb\Big( S_{-\xi^{(i)}}(\chi_{x-1})-(x-1) =y+1 \Big).
\ea
The value $S_{-\xi^{(i)}}(\chi_{x-1})-(x-1)$
is the overshoot of the random walk $S_{-\xi^{(i)}}$ above the level $x-1$.

Let $H^{(i)}$ be the strictly ascending ladder height of the random walk $S_{-\xi^{(i)}}$.

Observe that the overshoot of the random walk  $S_{-\xi^{(i)}}$ above the level $x-1$
equals in distribution to the
the overshoot of the random walk $S_{H^{(i)}}$ above the level $x- 1$.
From Assumption $\mathbf{A}_1$ we have $\E H^{(i)} <\infty$,
see T1 in Section 18 in Spitzer \cite{spitzer2013principles}.

Therefore by formula (6.7) from Theorem 6.2, \S2.6, in Gut \cite{Gut-09} and equality \eqref{e:B}
above we get that for
any $y \in\mathbb N \cup \{0\}$
\ba
\lim_{x\to +\infty} \Pb\Big(\cX(\sigma)=(-y,i) \Big|  \cX(0)=(x,i)\Big)= \frac{\Pb(H^{(i)}\geq y+1)}{\E H^{(i)}},
\ea
and in particular for $y=0$
\ba\label{e:B1}
\lim_{x\to +\infty} \Pb\Big(\cX(\sigma)=(0,i) \Big|  \cX(0)=(x,i)\Big)= \frac{ 1}{\E H^{(i)}}>0.
\ea
Since the
probabilities $\Pb(\cX(\sigma)=(0,i) | \cX(0)=(x,i))$ are obviously positive for any $x\in\mbN$, equation \eqref{e:B1} implies that
\ba
\min_{i=1,\dots,m} \inf_{x\in\mathbb N} \Pb\Big(\cX(\sigma)=(0,i) \Big| \cX(0)=(x,i)\Big)>0,
\ea
so that \eqref{eq:reccurence469} holds true.


2. Existence of the stationary distribution for the Markov chain $\cX_\text{entrance}$ follows from the observation that
\ba
\cX_\text{entrance}(\iota)=\cX(\tau_\iota)=\cX(\sigma_\iota+1),\quad \cX(\sigma_\iota)=\cX_\text{exit}(\iota)
\ea
and the existence of the stationary distribution for $\cX_\text{exit}$.
\end{proof}

Note that the stationary distributions  $\pi_\text{exit}$ and $\pi_\text{entrance}$
are related as follows:
\ba
\label{eq:exit-entrance}
\Pb_{\pi_\text{entrance}}(\cX(\sigma)\in A)& = \Pb_{\pi_\text{exit}}(\cX(0)\in A), \\ 
\Pb_{\pi_\text{exit}}(\cX(1)\in A)         & = \Pb_{\pi_\text{entrance}}(\cX(0)\in A), \quad A\subseteq \mbZ.
\ea
see a 
detailed investigation of exit and entrance chains constructed for the general Markov chain in Mijatovi\'c and Vysotsky \cite{mijatovic2018stationary}.
\begin{lem}
\be
\label{eq:fin_exit}
{\E}_{\pi_\emph{\text{exit}}}|R(0)|=\sum_{x\geq 0}x\, \pi_\emph{\text{exit}}(-x) <\infty,  
\ee
\be 
\label{eq:fin_entrance}
{\E}_{\pi_\emph{\text{entrance}}}R(0)=\sum_{x\geq 1}x\, \pi_\emph{\text{entrance}}(x)<\infty.  
\ee
\end{lem}
\begin{proof}
It follows from {\bf A}$_3$ 
and \eqref{eq:exit-entrance} that it suffices to prove \eqref{eq:fin_exit} only.

With the help of formula (18) in Vysotsky \cite{vysotsky2015limit} we obtain an estimate for the expectation of an overshoot of a level:
for each $\alpha>0$ there is $K(\alpha)>0$ such that for all $(x,i)\in \cN^m$
\ba
\label{eq:est_mom_overshoot}
\E_{x,i}|R(\sigma_1)|\leq K(\alpha)+\alpha x.
\ea
Assume that $R(0)=0$. Denote 
\ba
x_n&:=\E|R(\sigma_n)|=\E|R_\text{exit}(n)|,\\
y_n&:=\E R(\tau_n)=\E R_\text{entrance}(n),\ n\geq 1,
\ea
where $R_\text{exit}$ and $R_\text{entrance}$ are the radii of the embedded Markov chains $\cX_\text{exit}$ and $\cX_\text{entrance}$ respectively.
Then \eqref{eq:est_mom_overshoot} and {\bf A}$_3$ yield inequalities
\ba
y_{n+1} \leq C(1+x_n),\quad x_{n+1}\leq K(\alpha)+\alpha  y_{n+1}.
\ea
Take some $\alpha\in (0,C^{-1})$. It can be shown by induction that
\ba
\sup_n x_n\leq \frac{\alpha C +K(\alpha)}{1-\alpha C}.
\ea
By the ergodic theorem, for any initial distribution we have convergence 
\ba
\lim_{n\to\infty}\frac{\sum_{k=1}^n |R_\text{exit}(k)|}{n}={\E}_{\pi_\text{exit}}|R(0)|\ \ \text{a.s.}
\ea
Hence, the Fatou lemma yields that
\ba
{\E}_{\pi_\text{exit}}|R(0)|\leq \frac{\alpha C +K(\alpha)}{1-\alpha C}<\infty,
\ea
and the Lemma is proved.
\end{proof}

Now we are ready to formulate our results for scaling limits of $\cX$.

Consider a collection of sequences $\{\cX^n(k)\}_{k\geq 0}=(R^n(k),l^n(k))_{k\geq 0}$ 
that have the same transition probabilities as $\{\cX(k)\}_{k\geq 0}$, but maybe different initial conditions
$\cX^n(0)\in\cZ^m$.
The stopping times $\sigma^n(k)$ and $\tau^n(k)$ are defined respectively.

Let
\ba
\label{e:329}
X^n(t)&:= \Big(\frac{{R^n([nt])}}{v_1\sqrt{n}}\1_{l^n([nt])=1},\dots, \frac{{R^n([nt])}}{v_m \sqrt{n}}\1_{l^n([nt])=m}\Big),\quad  t\geq 0,
\ea
be a $\mathbb R^m$-valued stochastic process that lives on the coordinate axes. Note that its coordinates may take
negative values.

Let $W_\textbf{p}(\cdot,x)$ be a Walsh Brownian motion started from 
$x\in E^m$, see \eqref{e:Em}, with weights $\mathbf{p}=(p_1,\dots,p_m)$ and the symmetric local time $L(\cdot,x)$ at 0, 
see the definition of the local time in \eqref{e:loctime}.
 
\begin{thm}
\label{thm:Walsh_lim_pertRW}
Assume that conditions {\bf A}$_1$--{\bf A}$_4$ are satisfied and
\ba
\label{e:x}
X^n(0) \Rightarrow x \in E^m.
\ea
Then we have convergence: 
\ba
(X^n, L^n)\Rightarrow (W_\emph{\textbf{p}}(\cdot,x), L(\cdot,x)),\  n\to\infty,
\ea
where
\begin{equation}
\label{eq:L338}
L^n(t)=\frac{1}{\sqrt n}\sum_{\tau_k^n\leq [nt]} \frac{R^n(\tau_k^n)-R^n(\sigma_{k+1}^n) }{v_{l^n(\tau_k^n)}},\quad t\geq 0,
\end{equation}
and 
\ba
\label{eq:Walsch_probab}
p_k
&= \frac{{\E}_{\pi_\text{\rm entrance}}R(0)v_k^{-1}\1_{l(0)=k}-{\E}_{\pi_\text{\rm exit}}R(0)v_k^{-1}\1_{l(0)=k}}
{{\E}_{\pi_\text{\rm entrance}}R(0)v_{l(0)}^{-1}-{\E}_{\pi_\text{\rm exit}}R(0)v_{l(0)}^{-1}}\\
&=\frac{{\E}_{\pi_\text{\rm entrance}}(R(0)-R(\sigma))v_k^{-1}\1_{l(0)=k}}{{\E}_{\pi_\text{\rm entrance}}(R(0)-R(\sigma))v_{l(0)}^{-1}}\\
&=
\frac{
{\E}_{\pi_\text{\rm exit}}(R(1)v_{l(1)}^{-1} \1_{l(1)=k} - R(0)  v_{l(0)}^{-1}\1_{l(0)=k})}
{{\E}_{\pi_\text{\rm exit}}(R(1)v_{l(1)}^{-1}- R(0)  v_{l(0)}^{-1})}.
\ea
\end{thm}
 This Theorem will be proven in Section \ref{section2_3_2}.

\begin{remk}
Equality of expectations in \eqref{eq:Walsch_probab} follows from 
\eqref{eq:exit-entrance} and the definition of the embedded exit and entrance Markov chains.
\end{remk}

Let $X$ be a perturbed random walk with a membrane $\{-d,\dots,d\}$, $d\geq 0$, that is a Markov chain on $\mbZ$ such that
the following assumption holds true:

\noindent
\textbf{B}$_1$.
\ba
&\Pb\Big(X(1)=x+y \,\Big|\,  X(0)=x\Big)=\Pb( \xi_+=y),\quad x>d,\\
&\Pb\Big(X(1)=x+y \, \Big| \, X(0)=x\Big)=\Pb( \xi_-=y),\quad x<-d,
\ea
where the integer valued random variables $\xi_-$, $\xi_+$ satisfy
\ba
\E \xi_\pm=0,\quad v_\pm^2=\text{Var}\, \xi_\pm\in (0,\infty).
\ea
and generate 1-arithmetic random walks.

In order to specify the scaling limit of  $X$ analogously to Theorem \ref{thm:Walsh_lim_pertRW} we have to introduce the 
embedded entrance and exit Markov chains and entrance and exit stopping times for the Markov chain $X$.

We define 
\ba 
\label{eq:defn_of_exits_entra}
&\widetilde \sigma_0:=0,\\
&\widetilde \tau_\iota:=\inf\{k\geq\widetilde \sigma_\iota\colon \ |X(k)|>d \}, \quad \iota\geq 0,\\
&\widetilde \sigma_{\iota+1}=\inf\{k>\widetilde \tau_\iota\colon  X(k-1)> d,\ X(k)\leq d \quad \text{or}\quad  X(k-1)<- d,\ X(k)\geq -d\},\quad \iota\geq 0.
\ea
The difference between definition \eqref{eq:defn_of_exits_entra} and definition \eqref{e:sigma}, \eqref{e:tau} is twofold. First, when 
$X(\widetilde \sigma_\iota)\in\{-d,\dots,d\}$, 
the Markov chain may stay on the membrane $\{-d,\dots,d\}$ more than one step, and hence in general in this case we have 
$\widetilde \tau_\iota\geq \widetilde \sigma_\iota+1$. 
Second, when $X$ ``jumps over'' the membrane, i.e., when $X(\widetilde \sigma_\iota-1)>d$ and $X(\widetilde \sigma_\iota)<-d$ or vice versa, we get
$\widetilde \tau_\iota = \widetilde \sigma_\iota$.

We adapt the set of  Assumptions \textbf{A} of Theorem \ref{thm:Walsh_lim_pertRW}.


 \noindent
\textbf{B}$_2$.  For all $x\in \{- d,\dots,d\}$
\ba
\Pb( \widetilde \tau_0<\infty   |   X(0)=x)=1.
\ea

\noindent
\textbf{B}$_3$. 
 For all $x\in \{- d,\dots,d\}$
\ba
 \E\Big[|X(\widetilde \tau_0)|  \Big|  X(0)=x\Big]<\infty.
\ea

\noindent
\textbf{B}$_4$.  The states $\mbZ\backslash \{-d,\dots,d\}$  of the Markov chain $X$ communicate. %
 
\medskip 

As above, we introduce the embedded Markov chains $X_\text{exit}(\iota)=X(\widetilde \sigma_\iota)$, $\iota\geq 1$, and 
$X_\text{entrance}(\iota)=X(\widetilde \tau_\iota)$, $\iota\geq 0$. Similarly to Lemmas 
\ref{lem:exit_entrance_stationary} and \ref{eq:fin_exit} it can be shown 
that $X_\text{exit}$ and 
$X_\text{entrance}$ have unique stationary  distributions $\pi_\text{exit}$ and  $\pi_\text{entrance}$, respectively,
which are integrable.

Consider a collection of Markov chains $\{ X^n(\cdot)\}_{n\geq 1}$ that have the same transition probabilities as $X(\cdot)$ but maybe
have different initial values $X^n(0)$.

\begin{thm}
\label{thm:Skew_lim_pertRW}
Suppose that Assumptions {\bf B}$_1$--{\bf B}$_4$ are satisfied and
\ba\label{eq:conv_initial_cond_SBM-846}
 \vf\Big(\frac{X^n(0)}{ \sqrt{n}}\Big) \Rightarrow x \in \mbR,
\ea
where $\vf$ is defined in \eqref{e:phi}.
Then we have convergence to a skew Brownian motion $W^\text{\emph{skew}}_\gamma(\cdot,x)$ started at $x$:
\ba
\vf\Big(\frac{X^n(n\cdot )}{ \sqrt{n}}\Big)\Rightarrow W_{\gamma}^\text{\emph{skew}}(\cdot,x),\quad  n\to\infty,
\ea
where 
\ba
\label{eq:gamma_probab}
\gamma= 
\frac{{\E}_{\pi_\text{\emph{entrance}}}\vf(X(0)-X(\widetilde \sigma_1)) }{{\E}_{\pi_\text{\emph{entrance}}}|\vf(X(0)-X(\widetilde \sigma_1))|}
&=\frac{{\E}_{\pi_\text{\emph{entrance}}} (X(0)-X(\widetilde \sigma_1))v^{-1}_{\sgn(X(0)-X(\widetilde \sigma_1))}}
{{\E}_{\pi_\text{\emph{entrance}}} |X(0)-X(\widetilde \sigma_1)|v^{-1}_{\sgn(X(0)-X(\widetilde \sigma_1))} }\\
&=\frac{{\E}_{\pi_\text{\emph{entrance}}} (X(0)-X(\widetilde \sigma_1))v^{-1}_{\sgn X(0)}}
{{\E}_{\pi_\text{\emph{entrance}}} |X(0)-X(\widetilde \sigma_1)|v^{-1}_{\sgn X(0)} }.
\ea
\end{thm}

This Theorem will be proven in Section \ref{section:proof_skewBM}.

\begin{remk}
The case when the membrane $A$ is an arbitrary finite set $\{x_1,\dots,x_m\}\subset \mbZ$ 
is also covered by Theorem \ref{thm:Skew_lim_pertRW}. Indeed, $X$ can be considered as a random walk perturbed on bigger 
membrane $\{-d,\dots,d\}$, where $d=\max_{1\leq i\leq m}|x_i|$.
\end{remk}

\begin{remk}
\label{r:3.7}
 Condition $\mathbf{B}_3$ can be replaced with a simpler one:\\
$\mathbf{B}'_3$.  
\ba
\max_{|x|\leq d} \E\Big[|X(1)|\, \Big|\,  X(0)=x\Big]<\infty.
\ea
 Indeed, for any $|x|\leq d$ we have
\ba
\E\Big[|X(\tilde \tau_0)\,|\,X(0)=x\Big]
&= \sum_{|y|\leq d}\sum_{k=1}^\infty \P\Big(\widetilde \tau_0=k,X(k-1)=y\,|\,X(0)=x\Big)\E\Big[|X(k)|\,\Big|\,\widetilde \tau_0=k,X(k-1)=y\Big]\\
&= \sum_{|y|\leq d}\sum_{k=1}^\infty \P\Big(\widetilde \tau_0=k,X(k-1)=y\,|\,X(0)=x\Big)\E\Big[|X(1)|\,\Big|\,\widetilde \tau_0=1,X(0)=y\Big]\\
&= \sum_{|y|\leq d}\sum_{k=1}^\infty \P\Big(\widetilde \tau_0=k,X(k-1)=y\,|\,X(0)=x\Big) \frac{\E[|X(1)|\1_{|X(1)|>d}\,|\,X(0)=y]}
{\P(|X(1)|>d\,|\,X(0)=y) }\\
&\leq  \sum_{|y|\leq d}
\frac{\E[|X(1)|\1_{|X(1)|>d}\,|\,X(0)=y]} {\P(|X(1)|>d\,|\,X(0)=y) }<\infty.
\ea
In the last line we use condition \textbf{B}$_2$ and the convention that $\frac{0}{0}=0$.
\end{remk}

 \begin{corl}
\label{c:sobm}
The weak limit of the sequence
$ \Big\{\frac{X([n\cdot ])}{\sqrt{n}} \Big\}_{n\geq 1}$ is a process 
$(Y(t))_{t\geq 0}:=\Big(\vf^{-1}(W^\text{\emph{skew}}_\gamma(t))\Big)_{t\geq 0}$, which is a skew oscillating Brownian motion
that satisfies the SDE
\begin{equation}
\label{e:Xinf}
Y(t) = \int_0^t (v_-\1_{Y(s)< 0} + v_+\1_{Y(s)\geq  0})\,\di W(s)
+ \frac{\gamma (v_++v_-) + v_+-v_-}{\gamma (v_+ - v_-) + v_+ + v_-} \, L_0^{Y}(t),\quad t\geq 0,
\end{equation}
where 
\ba
L_0^Y(t)=\lim_{\ve\to0+}\frac{1}{2\ve}\int_0^t
\1_{|Y(s)|<\ve} \,\di \langle Y\rangle_s= \lim_{\ve\to0+}\frac{1}{2\ve}\int_0^t\Big(v_-^2\1_{ -\ve< Y(s)< 0} + v_+^2\1_{0\leq Y(s)<\ve}\Big)   \,\di s
\ea
is the symmetric semimartingale local time of $Y$ at 0. 
\end{corl}

A result similar to Theorem \ref{thm:Skew_lim_pertRW}  is also true for a perturbed random walk on a graph, where the limit process will be a Walsh Brownian motion. Let us describe the construction. For some $m\in\mbN$, consider a Markov chain $\cX$ on the state space 
\ba
\Big(\mbN\times\{1,\dots,m\}\Big)\cup \{0\}=:\cN^m\cup\{0\}.
\ea
The set $\cN^m\cup\{0\}$ consists of $m$ natural-valued rays with a common end-point 0.  
If $\cX(k)\in\mbN\times \{i\}$, 
we say that $\cX$ belongs to $i$-th ray.

We denote the radius process associated with $\cX(\cdot)$ by $R=(R(k))_{k\geq 0}$, i.e.,
\ba
R(k)=\begin{cases}
      x, \text{ if } \cX(k)=(x,i)\in\cN_m ,\\
      0, \text{ if } \cX(k)=0.
     \end{cases}
\ea

Let us formulate a set of assumptions concerning the dynamics of $\cX$.

\noindent
\textbf{C}$_1$. Let $(\xi_k^{(1)})_{k\geq 1},\dots,(\xi_k^{(m)})_{k\geq 1}$ be $m$ independent sequences of iid 
integer valued random variables with
\ba
\E\xi^{(i)}_k=0 \quad \text{and}\quad  v^2_i:=\operatorname{Var}\xi^{(i)}_k\in (0,\infty),\quad 1\leq i\leq m. 
\ea
We assume that for each $k\geq 0$, $\cX(k)$ is independent of the `future' $(\xi_{k+j}^{(1)})_{j\geq 1},\dots,(\xi_{k+j}^{(m)})_{j\geq 1}$.

On each ray $\mbN\times\{i\}$, $\cX$ behaves as a random walk with increments $\xi^{(i)}$, that is if $\cX(k)=(x,i) \in \cN^m$ and $x+\xi_{k+1}^{(i)}>0$, then 
\ba
\cX(k+1):=(x+\xi_{k+1}^{(i)}, i). 
\ea
If $\cX(k)=(x,i) \in \cN^m$ and $x+\xi_{k+1}^{(i)}=0$, then we set 
\ba
\cX(k+1):=0. 
\ea
Finally, we assume that if $\cX(k)=(x,i) \in \cN^m$ and $x+\xi_{k+1}^{(i)}<0$, then 
the conditional distribution 
\ba
\text{Law}\big(\cX(k+1)\,\big|\,\cX(k)=(x,i), x+\xi_{k+1}^{(i)}<0\big)
\ea
depends only on $i$ and $x+\xi_{k+1}^{(i)}$.

\noindent

\textbf{C}$_2$. The conditional distribution $\text{Law}(\cX(1)|\cX(0)=0)$ satisfies
\ba
&\P\big( R(1)\in\mbN \,\big| \, \cX(0)=0\big)=1,\\
&\E\big[ R(1) \,\big| \, \cX(0)=0\big]<\infty.
\ea
\textbf{C}$_3$.  There is a constant $C>0$ such that for all $(x,i)\in \cN^m$
\ba
\E\big[ R(k+1) \, \big| \,  \cX(k)=(x,i),\, x+\xi_{k+1}^{(i)}<0\big]\leq C(1+|x+\xi_{k+1}^{(i)}|).
\ea
\textbf{C}$_4$.  All the states $\cN^m$ of $\cX$ communicate.

We define the following stopping times:
\ba 
&\widehat \sigma_0:=0,\\
&\widehat\tau_{\iota}:=\inf\{k\geq \widehat \sigma_\iota\colon  X(k)> 0 \},\quad \iota\geq 0,\\
&\widehat \sigma_{\iota+1}:=\inf\{k\geq\widehat \tau_\iota\colon \  R(k)+\xi_{k+1}^{(i)}\leq 0 \}, \quad \iota\geq 1.
\ea
The random variable $\widehat \sigma_\iota$ is the $\iota$-th moment when $X$ jumps to 0 or ``wants to overjump 0''. 
Introduce the entrance Markov chain $X_\text{entrance}(\iota):=X(\widehat\tau_{\iota})$, $\iota\geq 1$, with the  
(unique) invariant distribution $\pi_\text{entrance}$.
\begin{thm}
\label{thm:Walsh_non_rigorous}
Assume that Assumptions $\mathbf{C}_1$--$\mathbf{C}_4$ hold true.
Then the Donsker scaling of $R$ converges in distribution to a Walsh Brownian motion $W_\mathbf{p}(\cdot)$
starting at 0:
\ba
\Big(\frac{R([n\cdot ])}{v_1\sqrt n}\1_{l(n)=1},\dots,\frac{R([n\cdot ])}{v_m\sqrt n}\1_{l(n)=m}\Big) \Rightarrow W_\mathbf{p}(\cdot),\quad n\to\infty.
\ea
The weights $\mathbf{p}=(p_1,\dots,p_m)$ of the Walsh Brownian motion are defined as follows:
\ba
p_k=  \frac{{\E}_{\pi_\text{\rm entrance}}\Big[\big(R(0)-R(\widehat\sigma_1)-\xi_{\widehat\sigma_1+1}^{(k)}\big)v_{k}^{-1}\1_{l(0)=k}\Big]}{{\E}_{\pi_\text{\rm entrance}}\Big[\big(R(0)-R(\widehat\sigma_1)-\xi_{\widehat\sigma_1+1}^{(l(0))}\big)v_{l(0)}^{-1}\Big]}.
\ea
\end{thm}

\section{Examples\label{s:exa}}
\begin{expl}
We work in the setting of Theorem \ref{thm:Skew_lim_pertRW}.
Consider perturbations of a simple symmetric random walk, i.e., assume that 
$\xi_+\stackrel{\di}{=}\xi_-\stackrel{\di}{=}\xi$, where $\Pb(\xi=\pm1)=\frac12$. Then    $v_\pm^2=1 $ and $\vf(x)=x$. 

\smallskip

\noindent
1.
Assume that $d=0$. Then the membrane consists of the single point $\{0\}.$ Notice that $X(\widetilde \sigma_1)=0$, so that
the stationary distribution $\pi_\text{exit}$ is trivial, $\pi_\text{exit}(0)=1$. Let the law of the jump from $0$
coincide with the law of a random variable $\eta$,  
\ba
\Pb(X(1)=j \, | \, X(0)=0)=\Pb(\eta=j), \ j\in\mbZ.
\ea
Recall that $\Pb(\eta=0)<1$ by assumption $\mathbf{B}_2$. Hence the 
entrance stationary distribution is equal to the distribution of the random walk at the moment of exit from $0$ 
\ba
\pi_\text{entrance}(j)=\Pb(X(\tau_0)=j \, | \, X(0)=0)=\frac{\Pb(\eta=j)}{\Pb(\eta\neq 0)}, \quad j\neq 0.
\ea
The permeability parameter $\gamma$ of the limit skew Brownian motion is 
\be
\label{eq:gamma_HS}
\gamma=\frac{\E \eta}{\E |\eta|}.
\ee
In particular, if $\Pb(X(1)=1\, | \, X(0)=0)=p$ and $\Pb(X(1)=-1\, | \, X(0)=0)=1-p$ or some $p\in(0,1)$
then
\be
\label{eq:gamma_HS1}
\gamma=2p-1.
\ee
Formulas \eqref{eq:gamma_HS1}, \eqref{eq:gamma_HS} appeared for the first time in the paper by Harrison and Shepp \cite{HShepp-81}
(formula \eqref{eq:gamma_HS} was only announced there).  

\noindent
2. Assume that $d\in \mbN$, but the random walk can exit from $\{-d,\dots,d\}$ only to the neighboring points $d+1$
or $-d-1$. Then the distribution $\pi_\text{entrance}$ is supported on the two point set $\{-d-1,d+1\}$.
Denote 
\ba
\label{e:45}
\alpha:=\Pb(X(\widetilde\tau_0)=d+1\, | \, X(0)=-d),\quad \beta:=\Pb(X(\widetilde\tau_0)=-d-1\, | \, X(0)=d).
\ea
The stationary distribution $\pi_\text{entrance}$ is calculated explicitly:
\ba
\pi_\text{entrance}(d+1)=\frac{\alpha}{\alpha+\beta},\quad  \pi_\text{entrance}(-d-1)=\frac{\beta}{\alpha+\beta}.
\ea
Notice that $X(\widetilde\sigma_1)=d$ if $X(0)=d+1$ and $X(\widetilde\sigma_1)=-d$ if $X(0)=-d-1$. Hence the permeability parameter 
$\gamma$ has the form
\ba
\label{:g47}
\gamma=\pi_\text{entrance}(d+1)-\pi_\text{entrance}(-d-1)=\frac{\alpha-\beta}{{\alpha+\beta}}. 
\ea
This result was already obtained in Pilipenko and Pryhod'ko \cite{pilipenko2012limit}.
\end{expl}

\begin{expl}
 We work in the setting of Theorem \ref{thm:Skew_lim_pertRW}.
Assume that a perturbed random walk on $\mbZ$ is such that 
$\Pb(\xi_+\geq -1)=\Pb(\xi_-\leq 1)=1$ and assume that $d\in\mbN$.
Then $\pi_\text{exit}$ is concentrated in two points  $\{-d\}$ and $\{d\}$. 
Denote
\ba
\widetilde \tau_0:= \inf\{k\geq 0\colon |X(k)|>d\}.
\ea
and
\ba
\alpha:=\Pb(X(\widetilde \tau_0)>0\, | \, X(0)=-d),\quad \beta:=\Pb(X(\widetilde \tau_0)<0\, | \, X(0)=d).
\ea
Then the stationary exit distribution $\pi_\text{exit}$ is
\ba
\pi_+:=\pi_\text{exit}(d)=\frac{\alpha}{\alpha+\beta}, \quad  \pi_-:=\pi_\text{exit}(-d )=\frac{\beta}{\alpha+\beta},
\ea
and the stationary entrance distribution $\pi_\text{entrance}$ is
\ba
\pi_\text{entrance}(x)= \pi_- \Pb(X(\widetilde \tau_0)=x \, | \, X(0)=-d) +\pi_+ \Pb(X(\widetilde \tau_0)=x \, | \, X(0)=d), \ |x|>d. 
\ea
Notice that $X(\widetilde \sigma_1)=d\, \sgn X(0)$  and $X(0)-X(\widetilde \sigma_1)= X(0)-d\, \sgn X(0) $ if $|X(0)|>d.$
Hence the parameter $\gamma$ from Theorem \ref{thm:Skew_lim_pertRW} equals
\ba
\gamma
=\frac{{\E}_{\pi_\text{entrance}} (X(0)-d\sgn X(0))v^{-1}_{\sgn X(0)}}
{{\E}_{\pi_\text{entrance}} |X(0)-d\sgn X(0)|v^{-1}_{\sgn X(0)} }.
\ea
In particular, if $v_+=v_-,$ then
\ba
 \gamma
=\frac{{\E}_{\pi_\text{entrance}} (X(0)-d\sgn X(0)) }
{{\E}_{\pi_\text{entrance}} |X(0)-d\sgn X(0)|  }.
\ea
If additionally \eqref{e:45} holds true, then 
\ba
\gamma=\frac{\alpha-\beta}{{\alpha+\beta}}. 
\ea
according to formula \eqref{:g47}.
\end{expl}

\begin{expl}
 \label{ex:4.3}
We work in the setting of Theorem \ref{thm:Walsh_lim_pertRW}.
Assume that  the random walk $\cX$ on $\cZ^m$, $m\geq 1$, is such that $\Pb( \xi^{(i)}=\pm1)=1/2$ and
\ba
\Pb(\cX(1)=(1,j)\, |\,  \cX(0)=(0,i))= q_{ij},
\ea
where $Q=\|q_{ij}\|$ is an irreducible stochastic matrix. We denote its stationary distribution (column vector)
by $\pi=(\pi_i)_{1\leq i\leq m}$ and recall that is the unique solution of the equation $Q^T\pi = \pi$.

Notice that the variances of the step sizes are equal to 1, $v_i^2 = 1$,
and the stationary distribution $\pi_\text{entrance}$
is supported on the set $\{(1,i)\}_{1\leq i\leq m}$. Moreover, $R(\sigma) = 0$, so
$\cX_\text{entrance}$, takes values in $\{1\} \times\{1, \ldots , m\}$
and has transition probabilities
\ba
\Pb(\cX_\text{entrance} (1) = (1, j) \,|\, \cX_\text{entrance}(0) = (1, i)) =
\Pb(\cX (1) = (1, j) \,|\, \cX (0) = (0, i)) = q_{ij} .
\ea
Therefore we have
\ba
\pi_\text{entrance}(1,i)=\pi_i,\quad 1\leq i\leq m.
\ea
Hence by \eqref{eq:Walsch_probab} the weights $\mathbf{p}=(p_1,\ldots,p_m)$ of the limit WBM
$W_\textbf{p}$ coincide with the stationary distribution $\pi=(\pi_1,\ldots,\pi_m)$.

Let $p_1, \ldots, p_m$ be arbitrary positive probabilities, $p_1 + \cdots + p_m = 1$.
Assume that $q_{ij} = p_j$ for all
$i, j = 1, \ldots , m$,
i.e., the choice of the next label is independent of the previous label. Then $\pi_i = p_i$,
$i = 1,\ldots , m$.
To the best of our knowledge, this convergence of a perturbed Markov chain to a WBM with
\emph{arbitrary} weights is a new result.
\end{expl}

 \begin{expl}
 \label{ex:4.4}
The model from Example \ref{ex:4.3} is related to that studied in Section 3 in Enriquez and Kifer
\cite{EnriquesKifer}.
Let us describe it in detail and show how it can be recovered from Theorem \ref{thm:Walsh_lim_pertRW}
and the methods used
in the proof. Instead of considering a planar spider graph consisting of $m$ rays, we will work on the space
$E^m$ defined in \eqref{e:Em}. Let $\{e_1 , \ldots , e_m\}$
be unit coordinate vectors and assume that $E^m$ is equipped with a
natural graph distance.
For $\ve>0$ consider a Markov chain $X_\ve=(X_\ve(k))_{k\geq 0}$ on $E^m$ whose jumps length are equal to $\ve $
and the transition probabilities are defined by the following rules:
\begin{enumerate} 
\item if $X_\ve(k)=r e_i$ and $r\geq \ve$, then $X_\ve(k+1)=(r\pm \ve) e_i $ with probability $1/2$;
\item if $X_\ve(k)=0$, then  $X_\ve(k+1)= \ve e_j$, $1\leq j\leq m$,  with probabilities $1/m$;
\item if $X_\ve(k)=r e_i$ and $r\in(0,\ve)$, then
\ba
X_\ve(k+1)=
\begin{cases}
(\ve-r)e_j,\ j\neq i,\\
(\ve+r)e_i,\ j=i,
\end{cases}
\ea
 with probability $q_{(j-i)\operatorname{mod} m}$, $1\leq i,j\leq m$,
 where $q_0,\ldots,q_{m-1}$ are positive probabilities,
 $q_0+\cdots +q_{m-1}=1$.
\end{enumerate}
Note that $X_\ve (k)$ can visit $0$ if and only if the initial value $X_\ve(0)$ belongs to the lattice
$(\ve\mathbb N^m )\cup \{0\}$.

Let a sequence of positive numbers $\{\ve_{n}\}$ be such that $\lim_{n\to\infty}\ve_n=0 $ and let
${X_{\ve_n}(0) }\Rightarrow x\in E^m$.

If $\ve^{-1}_n X_{\ve_n} (0) \in \mathbb N^m\cup \{0\}$ for every $n$, then Case 3 above never occurs, and by Theorem \ref{thm:Walsh_lim_pertRW} we have
convergence
\ba
\label{e:co1}
 {X_{\ve_n}(\ve_n^{-2} \cdot)}\Rightarrow W_{\textbf{p}}(\cdot,x),\  n\to\infty,
\ea
where $\textbf{p}=(\frac{1}{m},\dots,\frac{1}{m})$.

Assume now that the initial value
$\ve^{-1}_n X_{\ve_n} (0)$ does not belong to the lattice $\mathbb N^m\cup \{0\}$ for every $n$.
Then our
results are not applicable directly since the values $\ve^{-1}_n X_{\ve_n} (\cdot)$
are not integer. However, they are applicable
after some modifications.

We define a Markov chain $\cX^n = (R^n (k), l^n(k))_{k\geq 0}$
on $\mathcal Z^m = \mathbb Z \times \{1,\ldots , m\}$
as follows:
\ba
&\text{if } X_{\ve_n} (k) = r e_i \text{ with }r \in (x\ve_n, (x + 1)\ve_n)
\text{ for some }x\in\mathbb N \cup \{0\} \text{ and } i = 1, \ldots , m,\\
&\text{then } \cX^n(k) = (x, i).
\ea

It can be seen that all $\{\cX^n\}$ are Markov chains with possibly different initial conditions but with the
same transition probabilities
\ba
&\Pb(\cX^n(k + 1) = (x\pm 1, i)\, |\, \cX^n(k) = (x, i)) =\frac12,\quad \text{if } x \in \mathbb N,\\
&\Pb(\cX^n(k + 1) = (0, j)\, |\, \cX^n(k) = (0, i)) =q_{(i-j)\operatorname{mod} m},\quad \text{if } i\neq j,\\
\ea
 and
\ba
&\Pb(\cX^n(k + 1) = (1, i)\, |\, \cX^n(k) = (0, i)) =q_0.
\ea
Notice that the distance between $X^n(k):=\ve_n (R^n(k)\1_{l^n(k)=1},\dots, R^n(k)\1_{l^n(k)=m})$
and $X_{\ve_n}(k)$ does not exceed $\ve_n$. Hence, the limits of $\{X_{\ve_n}(\ve_n^{-2} \cdot)\}$ and
$\{X^n(\ve_n^{-2} \cdot)\}$ are equal if at least one of them exists.

Introduce the stopping times
\ba
\label{e:stop}
\sigma_1^n&=\inf\{k\geq 0\colon R^n(k)= 0\},\\
\tau^n_{\iota}&=\inf\{k>\sigma^n_\iota\colon  R^n(k)=1\}, \\
\sigma^n_{\iota+1}&=\inf\{k>\tau^n_\iota\colon  R^n(k)=0\},\ \iota\in\mbN,
\ea
and entrance and exit Markov chains
\ba
\cX^n_\text{exit}(\iota)&:= \cX^n(\sigma^n_\iota), \\
\cX^n_\text{entrance}(\iota)&:= \cX^n(\tau^n_\iota), \quad \iota\geq 1.
\ea
Note that Markov chains $\{\cX^n\}$ do not visit the set $\{\ldots , -2, -1\} \times \{1, \ldots , m\}$
at all. It is clear that the
entrance and exit Markov chains have unique stationary distributions $\pi_\text{entrance}$
and $\pi_\text{exit}$ concentrated on
the sets $\{1\} \times \{1, \ldots , m\}$ and $\{0\} \times \{1, \ldots , m\}$ respectively.
It follows from symmetry reasons that these
distributions are uniform. Therefore the probabilities $p_k$ defined in
\eqref{eq:Walsch_probab}
satisfy
\ba
p_k=
\frac{\E_{\pi_\text{entrance}} (R(0) - R(\sigma))\1_{l(0)=k}}
{\E_{\pi_\text{entrance}} (R(0) - R(\sigma))}
=\frac{\E_{\pi_\text{entrance}} (1 - 0)\1_{l(0)=k}}
{\E_{\pi_\text{entrance}} (1 -0)}
=\pi_\text{entrance}(1,k)=\frac{1}{m}.
\ea

The Markov chains $\{\cX^n \}$ make transitions from the set
$\{0\} \times \{1, \ldots , m\}$ to the same set $\{0\} \times \{1, \ldots , m\}$
which is not allowed by the assumptions of Theorem \ref{thm:Walsh_lim_pertRW}.
However if we skip such transitions, then the new
Markov chains will satisfy assumptions of Theorem
\ref{thm:Walsh_lim_pertRW}. Let us formally explain this construction.

We introduce a sequence $\{\lambda_n (k)\}_{k\geq 0}$ as follows
\ba
\label{e:426}
\lambda_n (0) &:= 0,\\
\lambda_n (k) &:=
k - \sum_{j=1}^k\1(\cX^n (j - 1)\in  \{0\} \times \{1, \ldots , m\}, \cX^n(j) \in \{0\} \times \{1, \ldots , m\} ) ,
\quad k \geq  1.
\ea
The sequence $\lambda^n$ is a time homogeneous additive functional of a Markov chain $\cX^n$
that counts all the steps
of $\cX^n$ except those that start and end within the set $\{0\} \times \{1, \ldots , m\}$.
The inverse mapping
\ba
\lambda_n^{-1} (k) := \inf\{j \geq  0\colon \lambda_n(j) \geq  k\},
\quad  k\geq  0,
\ea
is a sequence of stopping times, so that the process
$k \to \cX^n( \lambda_n^{-1} (k)) =: \widehat\cX^n(k) =: (\widehat R^n (k), \widehat l^n(k))$
is a time
homogeneous Markov chain with transition probabilities
\ba
\Pb( \widehat\cX^n(1) = (x \pm 1, i) \,|\, \widehat\cX^n(0)
 = (x, i)) = \frac12,\quad
 \text{if } x \in \mathbb N,
\ea
and
 \ba
\Pb( \widehat\cX^n(1) = (1, j) \,|\, \widehat\cX^n(0) = (0, i)) =
 \Pb( \cX^n(\tau^n_1) = (1, j) \,|\, \cX^n(0) = (0, i)) .
\ea
In particular this implies that $\widehat R^n$ is a Markov chain on $\mathbb N_0$
with transition probabilities
 \ba
\label{e:t1}
\Pb( \widehat R^n(1) = x \pm 1 \,|\, \widehat R^n(0) = x) = \frac12,\quad
 \text{if } x \in \mathbb N,
\ea
and
 \ba
 \label{e:t2}
\Pb( \widehat R^n(1) = 1 \,|\, \widehat R^n(0) = 0) = 1 .
\ea
Theorem \ref{thm:Walsh_lim_pertRW} gives us convergence
\ba
\label{e:co2}
\Big(\ve_n \widehat R^n(\ve_n^{-2}\cdot )\1_{\widehat l^n(\ve_n^{-2})=1},\ldots, \ve_n \widehat R^n(\ve_n^{-2}\cdot )\1_{\widehat l^n(\ve_n^{-2})=m}\Big)\Rightarrow W_{\textbf{p}}(\cdot,x) ,\quad  n\to\infty,
\ea
with the weights $\textbf{p}=(\frac{1}{m},\dots,\frac{1}{m})$.

Convergence \eqref{e:co1} in this case will follow from \eqref{e:co2} and forthcoming
Corollary \ref{corl:composition ordinary} if show that the
number of skipped steps is negligible, i.e., for each $t\geq 0$
\ba
\label{e:433}
\Big|\ve_n^2\lambda_n ([\ve_n^{-2}t])- t\Big|\stackrel{\Pb}{\to} 0,
\quad  n\to\infty.
\ea

Let
\ba
\widehat N^n(k) = \sum_{j=0}^k \1_{\widehat R^n(j) =0}
\ea
be the number of visits of $\widehat\cX^n$ to the set  $\{0\} \times \{1, \ldots , m\}$.
Since $\widehat R^n$
is a simple symmetric random walk on
$\mathbb N$ with reflection at $0$, see \eqref{e:t1} and \eqref{e:t2},
it is well known (see, e.g., Example 2 in Section 9 of Chapter 7
in Shiryaev \cite{Shiryaev-2-19}) that
 \ba
\ve_n^2\widehat N^n ([\ve_n^{-2}t])\stackrel{\Pb}{\to} 0,
\quad  n\to\infty.
\ea
Denote the time spent by $\cX^n$ in the set $\{0\} \times \{1, \ldots , m\}$  by
\ba
\zeta^n_\iota :=  \tau^n_\iota - \sigma^n_\iota,
\ea
see \eqref{e:stop}. The random variables $\{\zeta^n_\iota\}_{\iota\geq 0}$
are iid with the geometrical distribution with parameter $q_0$.

We recall \eqref{e:426} and estimate the number of transitions of $\cX^n$ inside of the set
$\{0\} \times \{1, \ldots , m\}$ as follows
\ba
k - \lambda_n (k)
\leq
\sum_{\iota=0}^{\widehat N^n(k)}
\zeta^n_\iota,
\ea
so that
\ba
\Big|\ve_n^2\lambda_n ([\ve_n^{-2}t])- t\Big|\leq \ve_n^2 + \ve_n^2\sum_{j=0}^{\widehat N^n([\ve_n^{-2}t])}
\zeta^n_j
\ea
and \eqref{e:433} is proved.

Summarizing all our findings, we get the convergence
\ba
 {  X_{\ve }(\ve^{-2} \cdot)}\Rightarrow W_{\textbf{p}}(\cdot,x),\quad \ve\to 0,
\ea
for any initial conditions such that $X_{\ve }(0)\Rightarrow x\in E^m$, and $\mathbf{p}=(\frac1m,\ldots,\frac1m)$.
\end{expl}

\begin{expl}
\label{ex:4.5}
In the work \cite{NgoPeigne19}, Ngo and Paign\'e considered a time homogeneous Markov chain $\{Y(k)\}_{k\geq 0}$ defined by 
\ba
\label{eq:Ngo}
Y(k+1)=\begin{cases}
Y(k)+\xi_{k+1}^+,&  \ \text{if}\ Y(k)>0\  \text{and} \ Y(k)+\xi_{k+1}^+> 0,\\
Y(k)+\xi_{k+1}^-,&  \ \text{if}\ Y(k)<0\  \text{and} \ Y(k)+\xi_{k+1}^-< 0,\\
\eta_{k+1}, & \ \text{if} \ Y(k)=0,\\
0, & \ \text{otherwise},
\end{cases}
\ea
where $\xi_k^+\stackrel{\di}{=}\xi_+$, $\xi_k^-\stackrel{\di}{=}\xi_-$, $\eta_k$, $k\geq 1$, are mutually independent random variables that are independent
of $Y(0)$. It is assumed that $\xi_-$ and $\xi_+$ are zero-mean integer-valued random variables with variances
$v^2_\pm=\operatorname{Var} \xi_\pm \in (0,\infty)$, and $\E|\eta|<\infty$.

Since $Y$ is forced to hit 0 at any attempt to jump over the origin, it does not satisfy the setting of 
Theorem \ref{thm:Skew_lim_pertRW}. However it can be fit into our theory with the help of a Markov chain 
$\cX(k)=(R(k), l(k))$, $k\geq 0$, on the state space $\mbZ\times\{-,+\}$ defined as follows.

We set
\ba
R(k+1)&=\begin{cases}
R(k)+\xi_{k+1}^+& \text{ if } l(k)=+ \text{ and } R(k)>0,\\
R(k)-\xi_{k+1}^- & \text{ if } l(k)=- \text{ and } R(k)>0,\\
|\eta_k| &  \text{ if }   R(k)\leq 0,
\end{cases}\\
l(k+1)&=\begin{cases}
l(k)&  \text{ if }   R(k)> 0,\\
+ &  \text{ if }   R(k)\leq 0 \text{ and } \eta_k \geq 0,\\
- &  \text{ if }   R(k)\leq 0 \text{ and } \eta_k < 0,
\end{cases}\\
R(0)&=|Y(0)|,\\
l(0)&=\begin{cases}
       +,\ \text{ if } Y(0)\geq 0,\\
       -,\ \text{ if } Y(0)< 0.
      \end{cases}
\ea
We have that $Y(k)=l(k)R(k)$ in each of the following three cases:
\ba
\Big\{Y(k)>0\  \text{and} \ Y(k)+\xi_{k+1}^+> 0\Big\}\text{ or }
\Big\{Y(k)<0\  \text{and} \ Y(k)+\xi_{k+1}^-< 0\Big\}\text{ or }
\Big\{Y(k)=0\Big\}.
\ea
Otherwise,
\ba
|Y(k)-l(k)R(k)|\leq |\xi^-_{k+1}| \vee|\xi^+_{k+1} |.
\ea
Since $\xi^\pm_{k}$ are square integrable we have
\ba
\frac{1}{\sqrt n} \max_{k\leq [nt]} \Big( |\xi^-_{k+1}| \vee|\xi^+_{k+1} |\Big)\stackrel{\P}{\to} 0,\quad n\to \infty, 
\ea
for any $t\geq 0$.

We apply Theorem \ref{thm:Walsh_lim_pertRW} to study the limit behaviour of $\cX$.
The stationary distribution $\pi_\text{entrance}$ of the Markov chain $\cX_\text{entrance}$ can be determined explicitly as 
\ba
\pi_\text{entrance}(x,+)&= \Pb(\eta=x \,|\, \eta\neq 0), \\
\pi_\text{entrance}(x,-)&= \Pb(\eta=-x \,|\, \eta\neq 0), \quad x\in\mbN. 
\ea
Let 
$\sigma=\inf\{k\geq 0\colon R(k)\leq 0\}$
be the time instant, when the radius  becomes non-positive for the first time. Denote by
\ba
f(x, \pm):= \E[R(\sigma)\, |\, (R(0),l(0))=(x,\pm)],\quad   x\in\mbN.
\ea
the expected value of the overshoot  over 0. Hence by the formula \eqref{eq:Walsch_probab}
\ba
\label{e:gp}
p_+
&=\frac{{\E}_{\pi_\text{\rm entrance}}(R(0)-R(\sigma))v_+^{-1}\1_{l(0)=+}}{{\E}_{\pi_\text{\rm entrance}}(R(0)-R(\sigma))v_{l(0)}^{-1}}
=\frac{\E (\eta-f(\eta,+))v_+^{-1}\1_{\eta>0}}
{\E (\eta-f(\eta,+))v_+^{-1}\1_{\eta>0}  +  \E (-\eta-f(-\eta,-))v_-^{-1}\1_{\eta<0}   },\\
p_-&=1-p_+.
\ea
Eventually, the Donsker scaling limit of $\varphi\Big(Y([n\cdot])/\sqrt n\Big)$, where $\varphi$ is defined in \eqref{e:phi},
is the skew Brownian motion with $\gamma =p_+-p_-,$ see Corollary \ref{thm:Walsh-skew}.

Our formula \eqref{e:gp} coincides with the result obtained in Ngo and Paign\'e \cite{NgoPeigne19} where the formula for $\gamma$ was characterized
in terms of descending ladder epochs 
and descending renewal functions
of the random walks with steps $\{\xi^+_k\}$ and  $\{\xi^-_k\}$, respectively.
\end{expl}

\begin{expl}
\label{ex:4.6}
We work in the setting of Theorem \ref{thm:Skew_lim_pertRW}  and assume that all the jumps starting from the membrane have mean value zero:
\ba
\E[X(1)\,|\,X(0)=x]=0,\quad |x|\leq d.
\ea
According to Corollary \ref{c:sobm}
the weak limit of 
$ \big(\frac{X([nt])}{\sqrt{n}} \big)_{t\geq 0}$ is a skew oscillating Brownian motion satisfying \eqref{e:Xinf}.
On the other hand, it is easy to see that the Markov chain $X=(X(k))_{k\geq 0}$ is a martingale,
and that the limit skew oscillating Brownian motion 
is a martingale, too. Therefore the local time term in \eqref{e:Xinf} vanishes. This means that the limit  of 
$ \big(\frac{X([nt])}{\sqrt{n}} \big)_{t\geq 0}$ is the 
oscillating Brownian motion that solves the SDE
\begin{equation}
\di Y(t) = (v_-\1_{Y(t)< 0} + v_+\1_{Y(t)\geq  0})\,\di W(t).
\end{equation}
This convergence was established by Helland in the case of a one-point membrane, $d=0$, see Corollary 8.4. in \cite{helland1982convergence}.
\end{expl}

\section{Proof of Theorem \ref{thm:Walsh_lim_pertRW} \label{section2_3_2}}
 
The proof of Theorem \ref{thm:Walsh_lim_pertRW} is divided into seven steps. 
 
\noindent
\textbf{Step 1.}
First we introduce several objects related to the random walk $\cX(k)=(R(k), l(k))$, $k\geq 0$.

We say that the radial part $R$ is in the `normal' mode at time $k$ if $R(k)>0$ and is in the `critical' mode otherwise. 
Hence, for $k\geq 0$, we denote by
\ba
\label{e:Tnormal}
T_\text{normal}(k):=\sum_{i=0}^{k-1} \1_{R(i)\in \mbN}, 
\ea
the number of jumps  in the normal mode up to time $k$, and 
\ba
\label{e:Tcritical}
T_\text{critical}(k):=\sum_{i=0}^{k-1} \1_{R(i)\leq 0}=k- T_\text{normal}(k), 
\ea
the number of jumps  in the critical mode, where $\sum_{i=0}^{-1}:=0$.
We also consider the inverses
\ba
T_\text{normal}^{-1}(\iota)&:=\min\{k\geq 0\colon T_\text{normal}(k)\geq \iota \},\\ 
T_\text{critical}^{-1}(\iota)&:=\min\{k\geq 0\colon  T_\text{critical}(k)\geq\iota \},\quad \iota\geq 1. 
\ea
Also note that 
\ba
\sigma_{\iota}=T_\text{critical}^{-1}(\iota) -1,\quad \iota\geq 1,
\ea
is a moment of $k$-th exit of $R$ from $\mbN$,  and 
\ba
\tau_{\iota}=T_\text{critical}^{-1}(\iota),\quad \iota\geq 1, 
\ea
see \eqref{e:sigma} and \eqref{e:tau} for the definition of $\tau_\iota$, $\sigma_\iota$.

Define random sequences $V$ and $U$ as follows:
\begin{align}
V(0)&:=0, \quad U(0):=0, \\
V(\iota)-V(\iota-1)&:=
\frac{R(T_\text{normal}^{-1}(\iota))}{v_{l(T_\text{normal}^{-1}(\iota))}}   -\frac{R(T_\text{normal}^{-1}(\iota)-1)}{v_{l(T_\text{normal}^{-1}(\iota)-1)}}
\\
\label{e:vv}
&=
\frac{R(T_\text{normal}^{-1}(\iota) )-R(T_\text{normal}^{-1}(\iota)-1)}{v_{l(T_\text{normal}^{-1}(\iota)-1)}}\\
&=\frac{R(T_\text{normal}^{-1}(\iota) )-R(T_\text{normal}^{-1}(\iota)-1)}{v_{l(T_\text{normal}^{-1}(\iota))}},
\  \iota\geq 1,
\\
\label{e:diffU}
U(\iota)-U(\iota-1)&:=
\frac{R(T_\text{critical}^{-1}(\iota))}{v_{l(T_\text{critical}^{-1}(\iota))}}-\frac{R(T_\text{critical}^{-1}(\iota)-1)}{v_{l(T_\text{critical}^{-1}(\iota)-1)}}
=\frac{R(\tau_{\iota})}{v_{l(\tau_{\iota})}}-\frac{ R(\sigma_{\iota})}{v_{l(\sigma_{\iota})}},\  \iota\geq 1.
\end{align}
In \eqref{e:vv} we used that 
the label $l$ does not change as long as $R$ is in the normal mode.

Notice that
\ba
\frac{R(k)}{v_{l(k)}}-\frac{R(k-1)}{v_{l(k-1)}}&=
\Big(V(T_\text{normal}(k) )-V(T_\text{normal}(k)-1)\Big)\1_{R(k-1)>0}\\
&+ \Big(U(T_\text{critical}(k))-U(T_\text{critical}(k)-1)\Big)\1_{R(k-1)\leq 0}.
\ea
Hence, for $k\geq 0$
\ba
\label{e:RVU}
\frac{R(k)}{v_{l(k)}}
&=\frac{R(0)}{v_{l(0)}} 
+\sum_{i=1}^{k}  \frac{R(i)-R(i-1)}{v_{l(i)}}\1_{R(i-1)>0}
+\sum_{i=1}^{k} \Big( \frac{R(i)}{v_{l(i)}} - \frac{R(i-1)}{v_{l(i-1)}}\Big) \1_{R(i-1)\leq 0 } \\
&=\frac{R(0)}{v_{l(0)}}
+\sum_{\iota=1}^{T_\text{normal}(k)} (V(\iota)-V(\iota-1))
+\sum_{\iota=1}^{T_\text{critical}(k)} (U(\iota)-U(\iota-1)) \\
&=\frac{R(0)}{v_{l(0)}} +
V(T_\text{normal}(k)) + U(T_\text{critical}(k)).
\ea
Since $T_\text{normal}^{-1}(\iota)$ and $T_\text{critical}^{-1}(\iota)$, $\iota\geq 1$, are stopping times w.r.t.\ the filtration generated by 
$\cX$, the sequence $(V(k))_{k\geq 0}$ is a martingale. The 
sequences 
\ba
\Big(R(T_\text{critical}^{-1}(\iota)) , v_{l(T_\text{normal}^{-1}(\iota))}\Big)_{\iota\geq 0}\quad \text{and}\quad  
\Big(R(T_\text{critical}^{-1}(\iota) -1), v_{l(T_\text{critical}^{-1}(\iota) -1}\Big)_{\iota\geq 1}
\ea
are the embedded entrance and exit  Markov chains introduced in \eqref{e:EE}.

For $t\in[0,\infty)$ we
define
\ba
R(t)&:=R([t]),\quad 
V(t):=V([t]),\quad U(t):=U([t]), \quad l(t):=l([t]),\\
T_\text{normal}(t)&:= T_\text{normal}([t])+(t-[t]), \quad  T_\text{critical}(t):= t- T_\text{normal}([t])
\ea
so that
\ba
\frac{R(t)}{v_{l(t)}}=\frac{R(0)}{v_{l(0)}} +{V(T_\text{normal}(t))} + U(T_\text{critical}(t)).
\ea
Introducing similar notation for processes $\cX^n$ we get from \eqref{e:RVU} that
\be
\frac{R^n(nt)}{v_{l^n(nt)}\sqrt{n}}=\frac{R^n(0)}{v_{l^n(0)}\sqrt{n}} +\frac{V^n(T^n_\text{normal}(nt))}{\sqrt{n}} + \frac{U^n(T^n_\text{critical}(nt))}{\sqrt{n}}.
\ee
In the next steps we will show that the ``normal'' $V^n$-component converges to a Brownian motion whereas
the ``critical'' $U^n$-component converges to a local time process.
\begin{remk}
We have the equality
\ba
\frac{U^n(T^n_\text{critical}(nt))}{\sqrt{n}}= L^n(t),
\ea
where $L^n(\cdot)$ is defined in  \eqref{eq:L338}.
\end{remk}

\medskip
\noindent
\textbf{Step 2.} We prove the following Lemma.

\begin{lem}
\label{lem:conv:mart_Ct}
The following weak convergence
\be
\label{eq:conv_toW}
\Big(\frac{V^n(nt)}{\sqrt{n}}, \frac{U^n(\sqrt{n}t)}{\sqrt{n}}\Big)_{t\geq 0} 
\Rightarrow\Big( W(t), \mu t\Big)_{t\geq 0} , \quad n\to\infty,
\ee
holds true,
where $W$ is a standard Brownian motion, 
\ba
\label{e:mu}
\mu&=\E_{\pi_\text{\rm entrance}}\frac{R(0)}{v_{l(0)}}-\E_{\pi_\text{\rm exit}}\frac{R(0)}{v_{l(0)}}\\
&=\E_{\pi_\text{\rm exit}}\left(\frac{R(1)}{v_{l(1)}}-\frac{R(0)}{v_{l(0)}}\right)\\
&=\E_{\pi_\text{\rm entrance}}\left(\frac{R(0)}{v_{l(0)}}-\frac{R(\sigma)}{v_{l(\sigma)}}\right)\\
&=\E_{\pi_\text{\rm entrance}}\left(\frac{R(0)-R(\sigma)}{v_{l(0)}}\right)\in(0,\infty),
\ea
and $\sigma$ is defined in \eqref{e:ss}.
\end{lem}
\begin{proof}[Proof of Lemma \ref{lem:conv:mart_Ct}]
Since the second coordinate in of the limit in \eqref{eq:conv_toW} is non-random, it is sufficient to prove the 
coordinate-wise
convergence, see the generalization of Slutsky's theorem in Ressel \cite{ressel1982topological}. 

Convergence $\frac{V^n(nt)}{\sqrt{n}} \Rightarrow W(t)$, $n\to\infty$ follows from the functional central limit theorem for martingale differences,
see Theorem 18.2 in Billingsley \cite{billingsley2013convergence}.
Indeed, due to Assumption \textbf{A}$_1$ since $\operatorname{Var} \xi^{(j)}=v_j^2\in (0,\infty)$, for every $t\geq 0$ and $\ve>0$ the 
Lindeberg condition holds true, namely
\ba
\frac{1}{n}\sum_{k=1}^{[nt]} & \E \Big[(V^n(k)-V^n(k-1))^2\1_{|V^n(k)-V^n(k-1)|\geq \sqrt{n}\ve}\Big] 
\leq t\sum_{j=1}^m \E \Big[\frac{(\xi^{(j)})^2}{v_j^2}\1_{|\frac{|\xi^{(j)}|}{v_j}|\geq \sqrt{n}\ve}\Big]\to 0, n\to\infty.
\ea
Furthermore, for each $t\geq 0$
\ba
\frac{1}{n}\sum_{k=1}^{[nt] } \E (V^n(k)-V^n(k-1))^2= \frac{[nt]}{n}\to t,\quad n\to\infty,
\ea
what yields the result.

Now we prove the convergence of the second coordinate. 

1. First assume that the initial condition $\cX(0)$ does not depend on $n$.
Let the initial condition $\cX(0)=(R(0),l(0))$ be fixed. 
Recall the entrance and exit embedded Markov chains $\cX_\text{entrance}$ and $\cX_\text{exit}$
defined in 
\eqref{e:EE}. Representation \eqref{e:diffU} together with the ergodic theorem for Markov chains immediately imply 
that for each $t\geq 0$ the limit
\ba
\frac{U(\sqrt{n} t)}{\sqrt n}
= \frac{1}{\sqrt n}\sum_{k=1}^{[\sqrt n t]} \Big( U(k)-U(k-1)\Big)
=\frac{[\sqrt{n} t]}{\sqrt n}
\frac{1}{[\sqrt{n} t]}\sum_{k=1}^{[\sqrt n t]}\Big(\frac{R(\tau_k)}{v_{l(\tau_k)}} - \frac{ R(\sigma_k)}{v_{l(\sigma_k)}}\Big)
\to \mu t,\quad n\to\infty.
\ea 
holds almost surely.
The functional convergence follows from Theorem 2.15 c) (i) in Chapter VI of Jacod and Shiryaev \cite{JacodS-03} because all the functions are 
non-decreasing and the limit is continuous.

2. In the situation when the distribution of the initial value $\cX^n(0)$ depends on $n$, we deal with weak convergence in the the scheme of series.
Hence, certain modifications of the above argument have to be made. 

Consider the set $A:=\{0\}\times\{1,\dots,m\}$ and let $\theta^n:= \inf\{k\geq 1\colon R^n(\sigma^n_k)=0\}$. 
We write
\ba
\Pb\Big(\Big|\frac{1}{\sqrt{n}}\sum_{k=1}^{[\sqrt{n}t]}
&\Big(\frac{R^n(\tau^n_k)}{v_{l^n(\tau^n_k)}}-\frac{ R^n(\sigma^n_k)}{v_{l^n(\sigma^n_k)}}\Big)-\mu t\Big|>\ve\Big)\\
&\leq 
\Pb\Big(\Big|\frac{1}{\sqrt{n}}\sum_{k=1}^{\theta^n}
\Big(\frac{R^n(\tau^n_k)}{v_{l^n(\tau^n_k)}}-\frac{ R^n(\sigma^n_k)}{v_{l^n(\sigma^n_k)}}\Big) \Big|>\frac{\ve}{3}\Big)\\
&+
\Pb\Big(\Big|\frac{1}{\sqrt{n}}\sum_{k=\theta^n}^{[\sqrt n t ]+\theta^n}
\Big(\frac{R^n(\tau^n_k)}{v_{l^n(\tau^n_k)}}-\frac{ R^n(\sigma^n_k)}{v_{l^n(\sigma^n_k)}}\Big)-\mu t\Big|>\frac{\ve}{3}\Big)\\
&+
\Pb\Big(\Big|\frac{1}{\sqrt{n}}\sum_{k=[\sqrt n t]+1}^{[\sqrt n t ]+\theta^n}
\Big(\frac{R^n(\tau^n_k)}{v_{l^n(\tau^n_k)}}-\frac{ R^n(\sigma^n_k)}{v_{l^n(\sigma^n_k)}}\Big) \Big|>\frac{\ve}{3}\Big)\\
&=S_1^n+S_2^n+S^n_3.
\ea
It is easy to see that for any $\ve>0$ and $t>0$
\ba
S_2^n\leq 
\sup_{\text{Law}(\cX(0))\in M(A)} 
\lim_{n\to\infty}\Pb\Big(\Big|\frac{1}{\sqrt{n}}\sum_{k=1}^{[\sqrt{n}t]}
\Big(\frac{R(\tau_k)}{v_{l(\tau_k)}}-\frac{ R(\sigma_k)}{v_{l(\sigma_k)}}\Big)-\mu t\Big|>\frac{\ve}{3}\Big)=0.
\ea
It follows from \eqref{eq:reccurence469} that there is $N=N_\delta\in\mbN$ such that for all $n\geq 1$
\ba
\Pb(\theta^n>N)<\frac{\delta}{2}.
\ea
Hence we can estimate
\ba
S_1^n+S_3^n\leq \delta &+ \Pb\Big(\frac{1}{\sqrt{n}}\sum_{k=1}^{N}
\Big(\frac{|R^n(\tau^n_k)|}{v_{l^n(\tau^n_k)}}+\frac{ |R^n(\sigma^n_k)|}{v_{l^n(\sigma^n_k)}}\Big)>\frac{\ve}{3}\Big)\\
&+
\Pb\Big(\frac{1}{\sqrt{n}}\sum_{k=[\sqrt n t]+1}^{[\sqrt n t ]+N}
\Big(\frac{R^n(\tau^n_k)}{v_{l^n(\tau^n_k)}}+\frac{| R^n(\sigma^n_k)|}{v_{l^n(\sigma^n_k)}}\Big) >\frac{\ve}{3}\Big)\\
&\leq \delta 
+ \sum_{k=1}^{N} \Pb\Big(\frac{1}{\sqrt{n}} \frac{R^n(\tau^n_k)}{v_{l^n(\tau^n_k)}}>\frac{\ve}{6N}\Big)
+\sum_{k=1}^{N} \Pb\Big(\frac{1}{\sqrt{n}}\frac{ |R^n(\sigma^n_k)|}{v_{l^n(\sigma^n_k)}}> \frac{\ve}{6N} \Big)\\
&+ \sum_{k=[\sqrt n t]+ 1}^{[\sqrt n t]+N} \Pb\Big(\frac{1}{\sqrt{n}} \frac{|R^n(\tau^n_k)|}{v_{l^n(\tau^n_k)}}>\frac{\ve}{6N}\Big)
+\sum_{k=[\sqrt n t]+1}^{[\sqrt n t]+N} \Pb\Big(\frac{1}{\sqrt{n}}\frac{ |R^n(\sigma^n_k)|}{v_{l^n(\sigma^n_k)}}> \frac{\ve}{6N} \Big)\\
&\leq \delta 
+2N \sup_{k\geq 1} \Pb\Big( |R^n(\tau^n_k)|>\frac{\ve\sqrt{n}\min v_i}{6N}\Big)
+2N \sup_{k\geq 1} \Pb\Big( |R^n(\sigma^n_k)|>\frac{\ve\sqrt{n}\min v_i}{6N}\Big).
\ea
Therefore it is enough to prove that for any $c_1$, $c_2>0$
\begin{align}
\label{eq:small_first_jumps1033}
&\sup_{k\geq 1}\Pb(R^n(\sigma^n_k)\leq -c_1\sqrt n)<\frac{\delta}{8N},\\
\label{eq:tau}
&\sup_{k\geq 1}\Pb(R^n(\tau^n_k)>c_2\sqrt n)<\frac{\delta}{4N}
\end{align}
for $n$ large enough.

Let us show that for any $j=1,\dots,m$
\ba
\lim_{a\to\infty } \sup_{i\in\mathbb N} \Pb\Big(|R(\sigma)|\geq a\, \Big|\, R(0) =i, l(0)=j\Big) =0.
\ea
Let $\ve>0$. 
From \eqref{e:B} it follows that there are $i_0>0$ and $a_0>0$ such that
\ba
\sup_{i> i_0} \Pb\Big(|R(\sigma)|\, \geq a_0\, \Big|\, R(0) =i, l(0)=j\Big) \leq \ve.
\ea
Since $\sigma<\infty$ a.s.\ for any starting value $(i,j)$, there is $a_1>0$ such that
\ba
\max_{i=1,\dots, i_0} \Pb\Big(|R(\sigma)|\geq a_1\, \Big|\, R(0) =i, l(0)=j\Big) \leq \ve
\ea
and \eqref{eq:small_first_jumps1033} follows.

To demonstrate  \eqref{eq:tau} we note that
\ba
\Pb(R^n(\tau^n_k)>c_2\sqrt n)&=\sum_{x\leq 0,j=1,\dots,m} 
\Pb\Big(R^n(\sigma^n_k+1)>c_2\sqrt n\, \Big|\, R^n(\sigma^n_k)=x,l^n(0)=j\Big)\Pb(R^n(\sigma^n_k)=x, l^n(\sigma_k^n)=j)\\
&\leq \sum_{x\leq 0,j=1,\dots,m}\Pb\Big(R^n(\sigma^n_k+1)>c_2\sqrt n\, \Big|\, R^n(\sigma^n_k\Big)=x,l^n(0)=j)\Pb(R^n(\sigma^n_k)=x).
\ea
Let $C>0$ be the constant from \textbf{A}$_3$. Choose $c_1>0$ such that $Cc_1/c_2< \delta/(8N)$. Then we obtain
\ba\Pb(R^n(\tau^n_k)>c_2\sqrt n)
&\leq \sum_{x\leq 0,j=1,\dots,m}\Big(\frac{C(1+|x|)}{c_2\sqrt n}\wedge 1\Big) \Pb(R^n(\sigma^n_k)=x)\\
&\leq \sum_{x\leq -c_1\sqrt n,j=1,\dots,m}\Big(\frac{C(1+|x|)}{c_2\sqrt n}\wedge 1\Big) \Pb(R^n(\sigma^n_k)=x)\\
&+\sum_{-c_1\sqrt n<x\leq 0 ,j=1,\dots,m}\Big(\frac{C(1+|x|)}{c_2\sqrt n}\wedge 1\Big) \Pb(R^n(\sigma^n_k)=x)\\
&\leq \Pb(R^n(\sigma^n_k)\leq -c_1\sqrt n)+ \frac{C(1+c_1\sqrt n)}{c_2\sqrt n} \leq \frac{\delta}{4N}
\ea
for $n\to\infty$ by \eqref{eq:small_first_jumps1033}.
The Lemma is proven.
\end{proof}

\medskip
\noindent
\textbf{Step 3.} The main result of this Step is Theorem~\ref{thm:limit_W_L674} about convergence of the scaled radius process to the 
reflected Brownian motion.

We will use the following result on the deterministic convergence in the Skorokhod space 
which is a particular case of Theorem 1 from Pilipenko and Sarantsev
\cite{pilipenko2023boundary}.

\begin{thm}
\label{thm:PilipenkoSarantsev}
Assume that sequence of c\`adl\`ag functions $(f_n,g_n,h_n, T_{A_n}, T_{B_n})$ is such that
\begin{enumerate}
\item $f_n(t)=g_n(T_{A_n}(t))+ h_n(T_{B_n}(t))$, $t\geq 0$, where 
\be
T_{A_n}(t)=\int_0^t\1_{f_n(s)>0}\, \di s,  \quad 
T_{B_n}(t)=\int_0^t\1_{f_n(s)\leq 0}\, \di s
\ee
and the set $\{s\in[0,t]\colon f_n(s)>0\}$ consists of finitely many intervals for any $t\geq 0$;
\item $g_n(\cdot)\to g_0(\cdot)$ in $D([0,\infty),\mbR)$;
    \item there is a sequence $(\rho_n)$ of positive numbers such that  $\rho_n\to 0$ and  $h_n(\rho_n \cdot )\to h_0(\cdot)$ in $D([0,\infty),\mbR)$;
    \item $g_0$, $h_0$ are continuous functions and $h_0$ is strictly increasing;
    \item $g_0(0)\geq 0$, $h_n(0)=0$. 
\end{enumerate}
Then $(f_n)$ converges in $D([0,\infty),\mbR)$ to the solution of the Skorokhod reflection problem at zero for the function $g_0$,  i.e.,
\ba
f_n(t)\to g_0(t)+\max_{s\in[0,t]}(-g_0(s))\vee 0.
\ea
Moreover,
\ba
h_n(T_{B_n}(t))\to \max_{s\in[0,t]}(-g_0(s))\vee 0
\ea
in $D([0,\infty),\mbR)$.
\end{thm}

Let the processes $R^n$ and $L^n$ be from Theorem \ref{thm:Walsh_lim_pertRW}.
\begin{thm}
\label{thm:limit_W_L674}
We have convergence in distribution in $D([0,\infty),\mbR^2)$:
\ba
\Big(\frac{{R^n(nt)}}{v_{l^n(nt)}\sqrt{n}}, L^n(t)\Big)_{t\geq 0}\Rightarrow
 \Big(|\|x\|+W(t)|, L_t^0(|\|x\|+W|)\Big)_{t\geq 0}, \quad n\to\infty,
\ea
where $L^0(|\|x\|+W|)$ is the symmetric local time at 0 of the reflected Brownian motion $(|\|x\|+W(t)|)_{t\geq 0}$ and 
$x\in E^m$ is from \eqref{e:x}.
\end{thm}
\begin{proof}
It follows from Lemma~\ref{lem:conv:mart_Ct} and
the Skorokhod representation theorem that there are copies $(\wt V^n(\cdot), \wt U^n(\cdot))\overset{\di}{=} (V^n(\cdot), U^n(\cdot))$ 
and $ \wt W  \overset{\di}{=}W$ defined on the same probability space such that we have a.s.\ locally uniform convergence
\ba
\Big(\frac{\wt V^n(nt)}{\sqrt{n}}, \frac{\wt U^n(\sqrt{n}t)}{\sqrt{n}}\Big)_{t\geq 0}\to (\wt W(t), \mu t)_{t\geq 0},\quad n\to\infty,
\ea
and 
\ba
\frac{\wt R^n(0)}{v_{\wt l^n(0)}\sqrt{n}}\to \|x\|,\quad n\to\infty \text{ a.s.},
\ea
where $x\in E^m$ is from \eqref{e:x}.

The functions $t\mapsto \wt T^n_\text{normal}(nt)$ and $t\mapsto \wt T^n_\text{critical}(nt)$ increase 
if $\frac{\wt R^n(nt)}{v_{\wt l^n(nt)}\sqrt{n}}$ belongs to $(0,\infty)$ and $(-\infty,0]$, respectively.
Hence Theorem \ref{thm:PilipenkoSarantsev} applied to the functions
\ba
f_n(t)&=\frac{\wt R^n(nt)}{\sqrt{n} v_{\wt l^n(nt)}},\quad  
g_n(t)=\frac{\wt R^n(0)}{v_{\wt l^n(0)}} + \frac{\wt V^n(nt)}{\sqrt{n}},\quad  h_n(t)=\frac{\wt U^n(n t)}{\sqrt{n}},\\
T_{A_n}(t) &= \frac{\wt T^n_\text{normal}(nt)}{n}, \quad   T_{B_n}(t)= \frac{\wt T^n_\text{critical}(nt)}{n},\quad  \rho_n=\frac{1}{\sqrt{n}},
\ea
yields that the term
\ba
\frac{\wt R^n(nt)}{\sqrt{n} v_{\wt l^n(nt)}}
=\frac{\wt R^n(0)}{v_{\wt l^n(0)}\sqrt{n}} +\frac{\wt V^n(\wt T^n_\text{normal}(nt))}{\sqrt{n}} + \frac{\wt U^n(\wt T^n_\text{critical}(nt))}{\sqrt{n}}
\ea
converges with probability one in $D([0,\infty),\mbR)$ to the solution of the Skorokhod reflection problem for the process 
$\|x\|+\wt W(t)$, $t\geq 0,$ i.e., to
\ba
\|x\|+\wt W(t) + \max_{s\in[0,t]}(-\|x\|-\wt W(s))\vee0.
\ea
Moreover, we have convergence 
\ba
\label{e:Urefl}
\frac{\wt U^n(\wt T^n_\text{critical}(nt))}{\sqrt{n}}\to \max_{s\in[0,t]}(-\|x\|-\wt W(s))\vee0,\ n\to\infty.
\ea
In particular we have that
\ba
\label{e:VW}
\frac{\wt V^n(\wt T^n_\text{normal}(nt))}{\sqrt{n}} \to \wt W(t).
\ea
It is well known that the law of the solution of the Skorokhod reflection problem coincides with the law of the reflected Brownian motion, see, e.g.,
\S 1.3 in Pilipenko~\cite{pilipenko14book}. This means that
\ba
\Big(\|x\|+\wt W(t) + \max_{s\in[0,t]}(-\|x\|-\wt W(s))\vee0, \max_{s\in[0,t]}(-\|x\|-\wt W(s))\vee0\Big)_{t\geq 0}
\stackrel{\di}{=}\Big(|\|x\|+W(t)|, L_t^0(|\|x\|+W|) \Big)_{t\geq 0},
\ea
and Theorem \ref{thm:limit_W_L674} is proven.
\end{proof}

\medskip
\noindent
\textbf{Step 4.} 
We continue proving Theorem \ref{thm:Walsh_lim_pertRW}. 
Similarly to Step 1, we decompose the radius process $R$ into a
sum of ``normal'' and ``critical'' components coordinate-wise. This will allow us to prove convergence of the
process $X^n$ defined in \eqref{e:329}.
First assume for simplicity that the initial value $X(0)$ does not depend on $n$.
Consider projections of $R$ on the 
line $\mbZ\times\{j\}$, $j=1,\dots,m$:
\ba
R_j(k)&:= R(k)\1_{l(k)=j}, \\
R(k)&:= R_1(k)+\dots+R_m(k),
\ea
Similarly to \eqref{e:Tnormal} we define
\ba
T_{\text{normal},j}(k):=\sum_{i=0}^{k -1} \1_{R_j(i)\in\mathbb N}, 
\ea
the number of jumps made from the positive side of $j$-th ray, as well as the inverse mapping 
\ba
T_{\text{normal},j}^{-1}(\iota):=\inf\{k\geq 0\colon T_{\text{normal},j}(k)\geq  \iota\}. 
\ea
\begin{align}
V_j(0)&=0,\quad U_j(0):=0, \\
V_j(\iota)-V_j(\iota-1)&:=
\frac{R_j(T_{\text{normal},j}^{-1}(\iota))}{v_{l(T_{\text{normal},j}^{-1}(\iota))}}   -\frac{R_j(T_{\text{normal},j}^{-1}(\iota)-1)}{v_{l(T_{\text{normal},j}^{-1}(\iota)-1)}}\\
&=
\frac{R_j(T_{\text{normal},j}^{-1}(\iota))-R_j(T_{\text{normal},j}^{-1}(\iota)-1)}{v_j},\quad  \iota\geq 1.
\end{align}

Note that $T_{\text{normal},j}^{-1}(\iota)-1$
is a stopping time with respect to the natural filtration of the Markov chain
because for any $k\in\mathbb N_0$ we have
\ba
\{T_{\text{normal},j}^{-1}(\iota)-1\leq k\}
=\{T_{\text{normal},j}^{-1}(\iota)\leq k+1\}
=\{T_{\text{normal},j}(k+1)\geq \iota\}
=\Big\{\sum_{i=0}^k \1_{R_j(i)\in\mathbb N}\geq \iota\Big\}.
\ea
Observe that if $\{T_{\text{normal},j}^{-1}(\iota)-1= k\}$
then $R_j (T_{\text{normal},j}^{-1}(\iota)-1) = R_j (k) \in\mathbb N$. It follows from the strong
Markov property and the assumptions of the Theorem that $(V_j (\iota))_{\iota\geq 0}$, $1\leq j \leq m$,
are independent random
walks, which increments have distributions of $\frac{\xi^{(j)}}{v_j}$ respectively.

%
%
%
%

Observe that
\ba
\frac{R_j(k)}{v_j}&=\frac{R_j(k)\1_{l(k)=j}}{v_{l(k)}}\\
&=\frac{R(0)\1_{l(0)=j}}{v_j} 
+\sum_{i=1}^{k} \Big( \frac{R(i)\1_{l(i)=j}}{v_j}-\frac{R(i-1)\1_{l(i-1)=j}}{v_j}\Big)\1_{R(i-1)>0}\\
&+\sum_{i=1}^{k} \Big( \frac{R(i)\1_{l(i)=j}}{v_j} - \frac{R(i-1)\1_{l(i-1)=j}}{v_j}\Big) \1_{R(i-1)\leq 0 } \\
&=\frac{R_j(0)}{v_j} 
 +\sum_{\iota=1}^{T_{\text{normal},j}(k)} (V_j(\iota)-V_j(\iota-1))\\
&+\sum_{i=1}^{k} \Big( \frac{R(i)\1_{l(i)=j}}{v_j} - \frac{R(i-1)\1_{l(i-1)=j}}{v_j}\Big) \1_{R(i-1)\leq 0 } \\
\ea
Note that
\ba
R(i-1)\leq 0\quad \Leftrightarrow \quad i-1=\sigma_\iota \quad \text{for some $\iota$} 
\quad \Leftrightarrow \quad i=\tau_\iota \quad \text{for some $\iota$} 
\ea
and 
\ba
\tau_\iota\leq k \quad \Leftrightarrow \quad \sigma_\iota\leq k \quad \Leftrightarrow \quad \iota\leq  T_{\text{critical}}(k).
\ea
Therefore,
\ba
\label{e:rrr}
\frac{R_j(k)}{v_j}&
=\frac{R_j(0)}{v_j} +  V_j(T_{\text{normal},j}(k)) + U_j(T_{\text{critical}}(k)),
\ea
where
\ba
U_j(l)
=\sum_{\iota=1}^{l} \Big( \frac{R(\tau_\iota)}{v_j} \1_{l(\tau_\iota)=j}
- \frac{R(\sigma_\iota)}{v_j} \1_{l(\sigma_\iota)=j}\Big).
\ea
The sequences $k\mapsto U_j( T_{\text{critical}}(k))$
are non-decreasing for all $j=1,\dots,m$.

Notice that $V(\iota)$ and $U(\iota)$, $\iota\geq 1$, defined in \eqref{e:vv} and \eqref{e:diffU} satisfy the property 
\ba
\label{e:VT}
V(T_{\text{normal}}(k))= V_1(T_{\text{normal},1}(k))+\dots + V_m(T_{\text{normal},m}(k)),\\
\ea
\be
\label{eq:U_crit_sum895}
 U(T_{\text{critical}}(k)) = U_1(T_{\text{critical}}(k))+\dots + U_m(T_{\text{critical}}(k)).
\ee
When the initial value $X(0)$ depends on $n$, similarly to notations above we 
define the sequences $R^n_j$, $V^n_j$, $U^n_j$, etc.

\medskip
\noindent
\textbf{Step 5.} 
To prove convergence, we will utilize the martingale characterization of the WBM.
First, we show that the sequence 
\ba
\label{e:tri}
\Big\{\Big(\frac{R^n_j(n\cdot)}{\sqrt{n}}, \frac{ V_j^n(T^n_{\text{normal},j}(n\cdot))}{\sqrt{n}}, 
\frac{U^n_j(T^n_{\text{critical}}(n\cdot))}{\sqrt{n}}\Big)_{1\leq j\leq m}\Big\}_{n\geq 1} 
\ea
is weakly relatively compact in $D([0,\infty),\mbR^{3m})$ and its limit is continuous.

Due to \eqref{e:Urefl}, the sequence of processes $\{\frac{U^n(T^n_{\text{critical}}(n\cdot))}{\sqrt{n}}\}_{n\geq 1}$
converges in distribution to a continuous non-decreasing process. Equation \eqref{eq:U_crit_sum895} implies that the modulus of continuity of 
each non-decreasing process $t\mapsto U^n_j(T^n_{\text{critical}}(nt))/\sqrt n$ is dominated by the modulus of continuity of $t\mapsto U^n(T^n_{\text{critical}}(nt))/\sqrt{n}$.
Hence this sequence is  weakly relatively compact, and any  limit process  
 $\widehat U^\infty=(\widehat U^\infty_1(t),\dots, \widehat U^\infty_m(t))_{t\geq 0}$,
 is continuous in $t$  and each coordinate $t\mapsto \widehat U^\infty_j(t)$ is non-decreasing.
 
 By Donsker's theorem,
 \ba
\Big( \frac{ V^n_j(n\cdot )}{\sqrt{n}},\dots, \frac{ V^n_m(n\cdot)}{\sqrt{n}}\Big)\Rightarrow \Big(B_1(t),\dots, B_m(\cdot)\Big),\ n\to\infty,
 \ea
 where $B_1,\dots, B_m$ are independent Brownian motions. Since
 $T^n_{\text{normal},j}(t)-T^n_{\text{normal},j}(s)\leq t-s$, for $0\leq s\leq t$, the sequence of processes 
 $\{(\frac{  V^n_j(T^n_{\text{normal},j}(n\cdot ))}{\sqrt{n}})_{1\leq j\leq m}\}_{n\geq 1}$
 is weakly relatively compact, too, and any limit point  
 $M^\infty=(M^\infty_1(t),\dots, M^\infty_m(t))_{t\geq 0}$ is a continuous process.
Weak relative compactness of $\{\frac{R^n_j(n\cdot)}{ \sqrt{n}}\}_{n\geq 1}$ follows 
from \eqref{e:rrr}. Since all limits $M_j^\infty$ and $\widehat U_j^\infty$ are continuous,
the limits $R_j^\infty$ are also continuous, and the sequence \eqref{e:tri} is tight in $D([0,\infty),\mbR^{3m})$.

We have
\ba
\frac{{R^n(nt)}}{v_{l^n(nt)}\sqrt{n}}
=\sum_{j=1}^m \frac{{R^n_j(nt)}}{v_j\sqrt{n}}
&= \sum_{j=1}^m \frac{{R^n_j(0)}}{v_j\sqrt{n}}
+  \sum_{j=1}^m \frac{ V_j^n(T^n_{\text{normal},j}(nt))}{\sqrt{n}}
+  \sum_{j=1}^m 
\frac{U^n_j(T^n_{\text{critical}}(nt))}{\sqrt{n}}\\
&=\frac{{R^n(0)}}{v_{l^n(0)}\sqrt{n}} + \frac{ V^n(T^n_{\text{normal}}(nt))}{\sqrt{n}}
+ \frac{U^n(T^n_{\text{critical}}(nt))}{\sqrt{n}}.
\ea
Recall that by Theorem \ref{thm:limit_W_L674}, see \eqref{e:Urefl} and \eqref{e:VW},
$\frac{ V^n(T^n_{\text{normal}}(n\cdot ))}{\sqrt{n}}$ converges to a Brownian motion $W(\cdot)$,
$\frac{{R^n(nt)}}{v_{l^n(nt)}\sqrt{n}}$ converges to a reflected Brownian motion 
\ba
\Big(\|x\|+  W(t) + \max_{s\in[0,t]}(-\|x\|- W(s))\vee0\Big)_{t\geq 0},
\ea
and 
$\frac{U^n(T^n_{\text{critical}}(n\cdot))}{\sqrt{n}}$ converges to its local time at $0$ that will be denoted by $\nu(\cdot)$.

Hence, any condensation point
\ba
(R^\infty_j(\cdot), M^\infty_j(\cdot), \widehat U^\infty_j(\cdot))_{1\leq j\leq m} 
\ea
of \eqref{e:tri}
is a continuous process such that
\begin{enumerate}
\item
\begin{align}
\label{e:W}
\sum_{j=1}^m M^\infty_j(t)&=W(t)\\
\label{e:nu}
\sum_{j=1}^m \widehat U^\infty_j(t)&=\nu(t)
,\quad t\geq 0;
\end{align}
\item 
for each $1\leq j\leq m$
\ba
X^\infty_j(t):=\frac{R^\infty_j(t)}{v_j}=  x_j + M^\infty_j(t)+ \widehat U^\infty_j(t),\quad t\geq 0;
\ea
\item $t\mapsto \widehat U^\infty_j(t)$ are non-decreasing for $t\geq 0$, $\widehat U^\infty_j(0)=0,$ for each $j=1,\dots,m$.
\end{enumerate}
Since the process
\ba
\label{e:RBM}
\sum_{j=1}^m X^\infty_j(t),\quad t\geq 0,
\ea
is a 
reflected Brownian motion 
whose local time at 0 is $\nu(\cdot)$ we have
\ba
\int_0^\infty\1_{\sum_{j=1}^m X^\infty_j(t)>0} \,\di \nu(t) =0 \mbox{ and } \int_0^\infty\1_{\sum_{j=1}^m X^\infty_j(t)=0}\,\di t=0
\ea
with probability 1

Also note that $X^\infty_i(t)X^\infty_j(t)=0$ for every $i\neq j$ and $t\geq 0$. Since $\sum_{j=1}^m X^\infty_j(t)\geq 0$ we conclude that 
all $X^\infty_j(t)\geq 0$, $j=1,\dots,m$.

\medskip
\noindent
\textbf{Step 6.}
It is left to demonstrate that any condensation point $(X^\infty_j(\cdot))_{1\leq j\leq m}$ 
is a WBM with weights $p_1,\dots,p_m$ and the starting point $x$.
Without loss of generality, we assume that the sequence \eqref{e:tri} weakly converges to the condensation point
itself.

Let us check conditions of the martingale characterization of the Walsh Brownian motion, see Theorem \ref{thm:Walsh}. 
It remains to verify the following:
\begin{enumerate}
\item for any $j=1,\dots,m$ and $t\geq 0$
\be
\label{eq:local_times_LLN945}
\frac{\widehat U^\infty_j(t)}{\nu(t)}= \frac{\widehat U^\infty_j(t)}{\sum_{l=1}^m \widehat U^\infty_l(t)}  =p_j  \ \mbox{a.s.},
\ee
where $p_j$ are specified in \eqref{eq:Walsch_probab};
\item for any $j=1,\dots,m$ the process $M^\infty_j$  is a  continuous square integrable martingale 
with respect to the filtration generated by $(X^\infty,M^\infty, \widehat  U^\infty)$ with the predictable quadratic variation
\ba
\lg M^\infty_j\rg_t =\int_0^t\1_{X^\infty_j(s)>0}\, \di s.
\ea
\end{enumerate}

By the Skorokhod representation theorem, see Theorem 6.7 in Billingsley \cite{billingsley2013convergence},
we may assume without loss of generality that the sequence \eqref{e:tri} is defined on the same probability space and that with probability 1 
convergence 
\ba
\Big(\frac{R^n_j(nt)}{v_j\sqrt{n}}, \frac{V^n_j(T^n_{\text{normal},j}(nt))}{\sqrt{n}}, 
\frac{U^n_j(T^n_{\text{critical}}(nt))}{\sqrt{n}}\Big)_{1\leq j\leq m} 
\to (X^\infty_j(t), M^\infty_j(t), \widehat U^\infty_j(t))_{1\leq j\leq m} ,\quad n\to\infty,
\ea 
holds true locally uniformly in $t$.

1. We demonstrate \eqref{eq:local_times_LLN945}. 
Recall that \eqref{eq:conv_toW} in
Lemma \ref{lem:conv:mart_Ct} imply that
\ba
\sum_{l=1}^m \frac{U^n_l(\sqrt nt)}{\sqrt n} \to
\mu t ,\quad n\to\infty,
\ea
where $\mu$ is defined in \eqref{e:mu}.
By \eqref{e:Urefl} and \eqref{e:nu} we have:
\ba
\label{e:UU}
\sum_{l=1}^m\frac{ U^n_l\Big(\sqrt n\frac{T^n_{\text{critical}}(nt)}{\sqrt n}\Big)}{\sqrt n}
=\sum_{l=1}^m\frac{ U^n_l(T^n_{\text{critical}}(nt))}{\sqrt n} 
 \to 
\sum_{l=1}^m \widehat U^\infty_l(t)=\nu(t) 
,\quad n\to\infty,
\ea
with probability 1 locally uniformly in $t$. Therefore for any $t>0$ 
\ba
\label{e:572}
\frac{T^n_{\text{critical}}(nt)}{\sqrt n} \to \frac{\nu(t)}{\mu},\quad n\to\infty.
\ea
Analogously to the proof of \eqref{eq:conv_toW} in
Lemma \ref{lem:conv:mart_Ct} we obtain that
\ba
\label{e:573}
\frac{U^n_j(\sqrt nt)}{\sqrt n} \to
p_j\mu t ,\quad n\to\infty,
\ea
in probability locally uniformly in $t$, where $p_j$ is defined in \eqref{eq:Walsch_probab}.
With the help of \eqref{e:572} and \eqref{e:573} we conclude that
\ba
\label{e:Uj}
 \frac{U^n_j(T^n_{\text{critical}}(nt))}{\sqrt n}= \frac{ U^n_j\Big(\sqrt n\frac{T^n_{\text{critical}}(nt)}{\sqrt n}\Big)}{\sqrt n}
\to \widehat U_j^\infty(t)= p_j\nu(t).
 \ea
The limit \eqref{eq:local_times_LLN945} follows from \eqref{e:UU} and \eqref{e:Uj}.

2. To  study martingale properties of $M^\infty$ we need an auxiliary result about limits of local martingales.

Let $(\cF_t^\infty)_{t\geq 0}$ be the natural filtration generated by 
$(X^\infty_j(t), M^\infty_j(t), \widehat U^\infty_j(t))_{1\leq j\leq m}$, $t\geq 0$. 
It is easy to see that $W$ defined in \eqref{e:W}
is a $(\cF_t^\infty)$-Brownian motion. To complete the proof of the Theorem it suffices to show that 
\begin{equation}
\label{eq:w_j=int}
M^\infty_j(t)=\int_0^t\1_{X^\infty_j(s)>0}\, \di W(s)\ \ \text{a.s.},
\end{equation}
for any $t\geq 0$.
Note that
\ba
\frac{V^n_j(T^n_{\text{normal}, j}(nt))}{\sqrt{n}} = \int_0^t\1_{\frac{R^n_j(ns-)}{v_j\sqrt{n}}>0 }\, \di \frac{V^n(T^n_\text{normal}(ns))}{\sqrt{n}},
\ea
where this integral is understood as a stochastic integral w.r.t.\ to the square integrable martingal 
$(\frac{V^n(T^n_\text{normal}(nt))}{\sqrt{n}},\cF_t^n)_{t\geq 0}$, 
where $\cF^n_t=\sigma\{R^n_j(ns), V^n_j(T^n_{\text{normal},j}(ns)), \widehat U^n_j(ns),\, j=1,\dots, m,\, s\leq t\}$.

\medskip
\noindent
\textbf{Step 7.}
The proof of the Theorem will follow from the following Lemma. 

\begin{lem}
For all $j=1,\dots, m$ and $t\geq 0$ we have
\ba
\label{e:sko}
\lim_{n\to\infty}\E \Big|\int_0^t\1_{\frac{R^n_j(ns-)}{v_j\sqrt{n}}>0 }\, \di \frac{V^n(T^n_\text{\rm normal}(ns))}{\sqrt{n}} 
- \int_0^t\1_{X^\infty_j(s)>0}\, \di W(s)\Big|^2=0.
\ea
\end{lem}
\begin{proof} 
We note that for any $t_0>0$
\ba
&\E \Big|\int_0^{t_0}\1_{\frac{R^n_j(ns-)}{v_j\sqrt{n}}>0 }\, \di \frac{V^n(T^n_\text{\rm normal}(ns))}{\sqrt{n}}\Big|^2 \leq t_0,\\
&\E \Big|\int_0^{t_0}\1_{X^\infty_j(s)>0}\, \di W(s)\Big|^2  \leq t_0.
\ea
Therefore to prove \eqref{e:sko}
it is sufficient to check that for any $t_0\in(0,t]$
\ba
\label{e:sko-t0}
\lim_{n\to\infty}\E \Big|\int_{t_0}^t\1_{\frac{R^n_j(ns-)}{v_j\sqrt{n}}>0 }\, \di \frac{V^n(T^n_\text{\rm normal}(ns))}{\sqrt{n}} 
- \int_{t_0}^t\1_{X^\infty_j(s)>0}\, \di W(s)\Big|^2=0.
\ea
This will follow from Skorokhod's Theorem from \S 3 of Chapter II in \cite{Skorokhod_Issl}.
Recall the notation
\ba
X^n_j(t)=\frac{R^n_j(nt)}{v_j\sqrt{n}},\quad j=1,\dots,m,
\ea
and let 
\ba
M^n(t)= \frac{V^n(T^n_\text{\rm normal}(nt))}{\sqrt{n}}    ,\quad j=1,\dots,m,
\ea
To apply Skorokhod's Theorem we have to verify the following conditions:

\noindent
i) for any fixed $s\in[t_0,t]$
\ba
M_n(s)\to W(s)\quad \text{a.s.},\quad n\to\infty;
\ea

\noindent
ii) $\E |M_n(t)|^2\leq t$, $n\geq 1$;

\noindent
iii) for any fixed $s\in[t_0,t]$
\ba
\1_{X^n_j(s-)>0}\to \1_{X^\infty_j(s)>0}\quad \text{a.s.},\quad n\to\infty,
\ea

\noindent
iv) 
for any $\ve>0$
\ba
\label{e:584}
\lim_{h\to 0}\lim_{n\to\infty}\sup_{|t_2-t_1|< h}\P\Big(|\1_{X^n_j(t_2-)>0}-\1_{X^n_j(t_1-)>0}|>\ve\Big)=0.
\ea

Indeed, convergence i) was established in \eqref{e:VW}. Condition ii) holds true by the construction of the sequence $V^n$.  
We emphasize that in \cite{Skorokhod_Issl}, 
it was demanded that $\E |M_n(s)|^2\to s$, $n\to\infty$, for all $s\in[t_0,t]$. However the inspection of Skorokhod's argument shows
that only boundedness was used in the proof.

To demonstrate iii), we recall that $X^\infty=\sum_{l=1}^m X_l^\infty$ is a reflected Brownian motion. Hence with probability 1, $X^\infty(s)>0$
for any fixed $s>0$.
Therefore, for a.a.\ $\omega$ there exists $k=k(s,\omega)$ such that $X^\infty_k(s,\omega)>0$. 
Let $\kappa=\kappa(s,\omega)>0$ be such that 
\ba
\label{e:Xk}
\inf_{r\in[s-\kappa,s+\kappa]} X^\infty_k(r)>0.
\ea
Correspondingly, for $i\neq k$, we have 
\ba
\label{e:Xi}
X^\infty_i(r)=0\quad \text{on}\quad r\in[s-\kappa,s+\kappa].
\ea
Recall that $X^n_j(\cdot )\to X^\infty_j(\cdot)$ locally uniformly for all $j$ a.s.\ and in particular on the interval $[s-\kappa,s+\kappa]$.
Therefore, for $n$ large enough
\ba
\label{e:Xk1}
X^n_k(r)>0,\quad r\in[s-\kappa,s+\kappa]
\ea
and for all $i\neq k$
\ba
\label{e:Xi1}
X^n_i(r)=0,\quad r\in[s-\kappa,s+\kappa].
\ea
This yields iii).

Eventually we establish iv). First we note that by \eqref{e:Xk} and \eqref{e:Xi} the mapping $s\mapsto \1_{X^\infty(s)>0}$ is continuous in
probability on $[t_0,t]$, and thus is uniformly continuous in
probability on $[t_0,t]$, i.e., for each $\ve>0$
\ba
\lim_{h\to 0}\sup_{|t_2-t_1|< h}\P\Big(|\1_{X^\infty(t_2)>0}-\1_{X^\infty(t_1)>0}|>\ve\Big)=0
\ea
that is equivalent to 
\ba
\lim_{h\to 0}\sup_{|t_2-t_1|< h}\P\Big(\1_{X^\infty(t_2)>0}\neq \1_{X^\infty(t_1)>0}\Big)=0.
\ea
From \eqref{e:Xk}, \eqref{e:Xi}, \eqref{e:Xk1} and \eqref{e:Xi1} we get that for any $\delta>0$ and $s\in[t_0,t]$ there is 
$\kappa=\kappa(s)$ and $n_0=n_0(s)$ such that for all $j=1,\dots,m$
\ba
\sup_{|z-s|<\kappa(s)}\P\Big(\1_{X^n_j(z-)>0}\neq \1_{X^\infty_j(z)>0}\Big)<\delta, \quad n\geq n_0(s).
\ea
Since $[t_0,t]$ is a compact, from its infinite covering by the sets $(s-\frac12\kappa(s),s+\frac12\kappa(s))$, $s\in[t_0,t]$, we may choose a finite 
subcovering  $(s_p-\frac12\kappa(s_p),s_p+\frac12\kappa(s_p))$, $p=1,\dots,N$. Hence, for $|t_2-t_1|<h$ with
\ba
h:=\frac12\min_{p=1,\dots, N}  \kappa(s_p),
\ea
there is $p=p(t_1,t_2)$ such that
\ba
t_1,t_2\in (s_p- \kappa(s_p),s_p+ \kappa(s_p)).
\ea
Hence for all $n\geq \max\{n_0(s_1),\dots,n_0(s_N)\}$ we get
\ba
\P\Big(\1_{X^n_j(t_1-)>0}\neq \1_{X^n_j(t_2-)>0}\Big)
&\leq 
\P\Big(\1_{X^n_j(t_1-)>0}\neq \1_{X^\infty_j(t_1)>0}\Big)\\
&+\P\Big(\1_{X^\infty_j(t_1)>0}\neq \1_{X^\infty_j(t_2)>0}\Big)\\
&+\P\Big(\1_{X^\infty_j(t_2)>0}\neq \1_{X^n_j(t_2-)>0}\Big)\\
&\leq 2\delta
+\sup_{|s_2-s_1|< h}\P\Big(\1_{X^\infty(s_2)>0}\neq \1_{X^\infty(s_1)>0}\Big).
\ea
Therefore,
\ba
\limsup_{h\to 0} \lim_{n\to\infty}\sup_{|t_2-t_1|< h}\P\Big(|\1_{X^n_j(t_2-)>0}-\1_{X^n_j(t_1-)>0}|>\ve\Big)\leq 2\delta, 
\ea
what finishes the proof of \eqref{e:584} and \eqref{e:sko-t0}. The Lemma is proven.
\end{proof}

\section{Proof of Theorem \ref{thm:Skew_lim_pertRW}\label{section:proof_skewBM}}

Before we start proofs let us formulate two results on convergence of compositions of functions and  processes.
\begin{lem}\label{lem:composition ordinary}
Let $(f_n,\lambda_n)_{n\geq 0}$ be a sequence of c\`adl\`ag functions such that
\begin{enumerate}
      \item $f_n \to f_0, n\to\infty$ in $D([0,+\infty),\mbR^m)$;
      \item $(\lambda_n)_{n\geq 1}$ are non-negative and non-decreasing functions;
      \item$\lambda_n\to \lambda_0, n\to\infty,$ in $D([0,+\infty),\mbR)$;
      \item $\lambda_0$ is strictly increasing and continuous function.
\end{enumerate}
Then 
\be 
f_n\circ \lambda_n\to f_0\circ\lambda_0,\ n\to\infty,
\ee
in $D([0,+\infty),\mbR^m)$.
\end{lem}
\begin{proof}  
See Theorem 13.2.2, p.\ 430, in Whitt \cite{Whitt02}.
\end{proof}

\begin{corl}\label{corl:composition ordinary}
  Let $(\xi_n)_{n\geq 0}$ and $(\eta_n)_{n\geq 1}$ be sequences of stochastic c\`adl\`ag processes with values in $\mbR^m$ and $\mbR$ respectively.
Assume that 
  \begin{enumerate}
      \item $(\xi_n(t))_{t\geq 0}\Rightarrow (\xi_0(t))_{t\geq 0}$, $n\to\infty$ in $D([0,+\infty),\mbR^m)$;
      \item $(\eta_n)_{n\geq 1}$ are non-negative and non-decreasing processes;
      \item\label{cond3_lem_comp} $(\eta_n(t))_{t\geq 0}\Rightarrow (ct)_{t\geq 0}, n\to\infty$ in $D([0,+\infty),\mbR)$ where $c>0$ is a constant.
  \end{enumerate}
Then 
\be
\label{eq:conv_comp_ct}
(\xi_n(\eta_n(t))_{t\geq 0}\Rightarrow (\xi_0(ct))_{t\geq 0}, n\to\infty \text{ in } D([0,+\infty),\mbR^m).
\ee
\end{corl}

\begin{proof}
Since the process $ (ct)_{t\geq 0}$ is non-random, assumptions of the corollary are equivalent to the convergence
of pairs
\be
\label{conv}
(\xi_n(t),\eta_n(t))_{t\geq 0}\Rightarrow (\xi(t), ct)_{t\geq 0},
\ee
see Ressel \cite{ressel1982topological}.

The statement follows
from Lemma \ref{lem:composition ordinary} and the Skorokhod representation theorem, see for example Kallenberg \cite[Theorem 3.30, p.~56]{Kallenberg-02}.
\end{proof}
  
\begin{remk}
Since all $\eta_n$ are non-decreasing and the limit is non-random,  condition \ref{cond3_lem_comp} of Corollary \ref{corl:composition ordinary}
is equivalent to weak convergence (or convergence in probability) $\eta_n(t)\to ct$, $n\to\infty$ for any fixed $t\geq 0$.
\end{remk}  
  
The next lemma gives a 
reverse result: convergence of compositions yields convergence of functions. We formulate the result for functions and then a corollary for processes.
\begin{lem}  \cite[Lemma 3.5]{iksanov2023functional}
\label{lemma:randomWalkTimeChangeConvergence}
Let $(f_n)_{n\geq 0}\subset D([0,\infty),\mbR^m)$ and  $(\mu_n)_{n\geq 0} \subset D([0,\infty),\mbR)$, $\mu_n$ be non negative and non decreasing. Assume that, for all $T>0$,
\be
\label{lemma:randomWalkTimeChangeConvergence:lambda}
\lim_{n\to\infty}\sup_{t \in [0,\, T]} |\mu_n(t) - t| = 0
\ee
and
\be
\label{lemma:randomWalkTimeChangeConvergence:fConv}
\lim_{n\to\infty} f_n\circ \mu_n=f_0 
\ee
in the $J_1$-topology in $D([0,\infty),\mbR^m)$. For $n\in\mbN$, denote by $(t_k^{(n)})_{k}$ the points of discontinuity of $\mu_n$
and put $u_k^{(n)}:= \mu_n(t_k^{(n)}-)$ and $v_k^{(n)}:= \mu_n(t_k^{(n)})$. If, in addition to \eqref{lemma:randomWalkTimeChangeConvergence:lambda} and \eqref{lemma:randomWalkTimeChangeConvergence:fConv}, for all $T>0$,
\be \label{lemma:randomWalkTimeChangeConvergence:condition}
\lim_{n\to\infty} \sup_{k}\sup_{s \in [u_k^{(n)}, v_k^{(n)})\cap [0,\, T]} |f_n(s) - f_n(u_k^{(n)}-)|=0,
\ee
then $$\lim_{n\to\infty}f_n=f_0$$ in the $J_1$-topology in $D([0,\infty),\mbR^m)$.
\end{lem}
\begin{remk}
    If $\mu_n$ are continuous, $\mu_n(0)=0$,  and $\mu_n(\infty)=\infty,$ then the result of the lemma follows from Lemma \ref{lem:composition ordinary}. Indeed, let  $\lambda_n(t)=\mu_n^{-1}(t):=\inf\{s\geq 0\colon \mu_n(s)\geq t\}.$ Then $\mu_n\circ\lambda_n=\text{id}$. So $f_n=(f_n\circ\mu_n)\circ \lambda_n$. It can be easily seen that $\lambda_n$ locally uniformly converges to $t$, so all conditions of  Lemma \ref{lem:composition ordinary} are satisfied. New features appear if $\mu_n$ may be discontinuous, so the compositions $f_n\circ \mu_n$ do not use parts of $f_n$ paths corresponding to jumps of $\mu_n.$
\end{remk}

The proof of the next corollary is obtained with the help of Skorokhod's representation theorem similarly to the proof of Corollary  \ref{corl:composition ordinary}.

\begin{corl}
\label{corl:randomWalkTimeChangeConvergence}
  Let $(\xi_n)_{n\geq 0}$ and $(\eta_n)_{n\geq 1}$ be sequences of stochastic c\`adl\`ag processes with values in $\mbR^m$ and $\mbR$ respectively
  such that $\eta_n$, $n\geq 1$, are non negative and non decreasing. 
Assume that for all $T>0$, we have locally uniform convergence in probability
\be
\label{corl:randomWalkTimeChangeConvergence:lambda}
\sup_{t \in [0,T]} |\eta_n(t) - t|\toP 0, {n\to\infty},
\ee
and convergence in distribution
\be
\label{corl:randomWalkTimeChangeConvergence:fConv}
 \xi_n\circ \eta_n\Rightarrow\xi_0, \ {n\to\infty},\ \text{ in }D([0,\infty),\mbR^m).
\ee
For $n\in\mbN$, denote by $(t_k^{(n)})_{k}$ the points of discontinuity of $\eta_n$
and 
 put $u_k^{(n)}:= \eta_n(t_k^{(n)}-)$ and $v_k^{(n)}:= \eta_n(t_k^{(n)})$. If, in addition 
 to \eqref{corl:randomWalkTimeChangeConvergence:lambda} and \eqref{corl:randomWalkTimeChangeConvergence:fConv}, 
 for all $T>0$, we have the uniform convergence in probability
\be 
\label{corl:randomWalkTimeChangeConvergence:condition}
 \sup_{k}\sup_{s \in [u_k^{(n)}, v_k^{(n)})\cap [0,\, T]} |\xi_n(s) - \xi_n(u_k^{(n)}-)|\toP0, \ {n\to\infty},
\ee
then 
\ba
\xi_n\Rightarrow\xi_0, \ {n\to\infty},
\ea
in $D([0,\infty),\mbR^m)$.
\end{corl}

\begin{proof}[Proof of Theorem \ref{thm:Skew_lim_pertRW}]
 We will prove Theorem \ref{thm:Skew_lim_pertRW} in ``reverse'' order. In Step 1, we construct a Markov
chain $\cX^n$ on the enlarged state space $\mathbb Z \times \{-, +\}$
that mimics the dynamics of $X^n$ and converges to a WBM
due to Theorem \ref{thm:Walsh_lim_pertRW}.
In Step 2, we transform the state space  $\mathbb Z \times \{-, +\}$ into a $\mathbb Z\backslash\{-d,\ldots, d\}$
and construct
an auxiliary Markov chain $\widehat X^n$
whose Donsker scalings converge to a SBM. In Step 3 we define a random
time change that removes times that the original Markov chain $X^n$ spends on the membrane $\{-d,\ldots, d\}$,
thus obtaining the Markov chain $\widehat X^n$.
In the final Step 4 we show that this random time change does not
affect the convergence.

\medskip
\noindent
\textbf{Step 1.}
We will employ Theorem \ref{thm:Walsh_lim_pertRW} for the convergence to the WBM and construct Markov chains $\cX^n=(R^n,l^n)$, $n\geq 1$, on $\mathbb Z\times\{-,+\}$. The original Markov chain
$X^n$ on $\mbZ$ will be obtained from $\cX^n$ by two transformations.

The transition probabilities of $\cX^n$ do not depend on $n$ and are defined as follows. 
If $R^n(0)=x\in\mbN$, i.e., if $R^n(0)$ is on the positive side, then we set
\ba
&\Pb( (R^n(1), l^n(1)) = (y, +)\ |\ (R^n(0), l^n(0)) = (x, +) ):= \Pb(\xi_+=y-x),\\
&\Pb( (R^n(1), l^n(1)) = (y, -)\ |\ (R^n(0), l^n(0)) = (x, -) ):= \Pb(-\xi_-=y-x),
\ea
that is  the distribution of a jump is equal to the distribution of $\xi_+$ if the label is $+$ and $(-\xi_-)$ if the label is $-$.
 
Assume that $R^n(0)=x\in\mbZ\setminus\mathbb N$, i.e., the Markov chain $\cX^n$ starts on the membrane.

If $R^n(0)\leq -2d-1$, then we set 
\ba
\label{e:over}
R^n(1):=-R^n(0)-2d,\quad l^n(1):=-l^n(0),
\ea
In other words, the random walk changes its label and and the radius gets reflected symmetrically with respect to the point $-d$. 
Notice that the new radius lies in $\mbN$. 

Recall $\widetilde \tau_0$ defined in \eqref{eq:defn_of_exits_entra}. If $R^n(0)=x\in\{-2d,\dots,0\}$ so that $x+d\in\{-d,\dots,d\}$, and $y\in\mbN$,
then we set
\ba
&\Pb( (R^n(1), l^n(1)) := (y, +)\ |\ (R^n(0), l^n(0)) = (x, +) ):=
    \Pb(X(\widetilde \tau_0)=y+d\ | X(0)=x+d),\\ 
&\Pb( (R^n(1), l^n(1)) := (y, -)\ |\ (R^n(0), l^n(0)) = (x, +) ):= 
    \Pb(X(\widetilde \tau_0)=-y-d\ | X(0)=x+d), \\ 
&\Pb( (R^n(1), l^n(1)) := (y, -)\ |\ (R^n(0), l^n(0)) = (x, -) ):= 
    \Pb(X(\widetilde \tau_0)=-y-d\ | X(0)=-x-d),\\ 
&\Pb( (R^n(1), l^n(1)) := (y, +)\ |\ (R^n(0), l^n(0)) = (x, -) ):=
    \Pb(X(\widetilde\tau_0)=y+d\ | X(0)=-x-d).
\ea

We define the initial condition $\cX^n(0)=(R^n(0), l^n(0))$ as follows:
\ba
(R^n(0), l^n(0))=\begin{cases}
    (X^n(0)-d, +), & \text{ if } X^n(0)\geq -d,\\
    (-X^n(0)-d, -), & \text{ if } X^n(0)< -d,
\end{cases}
\ea
where $(X^n(k))_{k\geq 0}$ is the original perturbed random walk.

Then
\ba
 \Big(\frac{{R^n(0)}}{v_-\sqrt{n}}\1_{l^n(0)=-}, \frac{{R^n(0)}}{v_+\sqrt{n}}\1_{l^n(0)=+}\Big) \Rightarrow (x_-, x_+),\ n\to\infty,
\ea
where $x_+=x\vee0$, $x_-=(-x)\vee 0$, and $x$ is from \eqref{eq:conv_initial_cond_SBM-846}. 

 Theorem   \ref{thm:Walsh_lim_pertRW} and Corollary \ref{thm:Walsh-skew} imply that 
\ba
\vf\Big(\frac{l^n(n\cdot)R^n([n\cdot])}{\sqrt{n}}\Big)\Rightarrow W^\text{skew}_\gamma(\cdot, x), \quad n\to\infty,
\ea
where $\vf$ is defined in \eqref{e:phi}. The permeability parameter $\gamma$ equals
\ba
\label{e:gamma1}
\gamma&=\frac{{\E}_{\pi_\text{entrance}}(R(0)-R(\sigma))(v_+^{-1}\1_{l(0)=+} - v_-^{-1}\1_{l(0)=-} ) }{{\E}_{\pi_\text{entrance}}(R(0)-R(\sigma))v_{l(0)}^{-1}}\\
\ea
see Theorem \ref{thm:Walsh_lim_pertRW} and Corollary \ref{thm:Walsh-skew}. The stopping time $\sigma$ was defined in \eqref{e:sigma} and
the stationary measure $\pi_{\text{entrance}}$ was defined in Lemma \ref{lem:exit_entrance_stationary}. 

Therefore, we  also have the convergence
\ba
\vf\Big(\frac{l^n([n\cdot])(R^n([n\cdot])+d)}{\sqrt{n}}\Big)\Rightarrow W^\text{skew}_\gamma(\cdot, x), \quad n\to\infty.
\ea

\medskip
\noindent
\textbf{Step 2.} Let us construct a one-dimensional Markov chain $\widehat X^n$ that converges to a SBM.

Note that the mapping $(R,l)\mapsto l\cdot (R+d)$ is a bijection of the sets $\mathbb N\times\{+\}$ and $\{d+1,d+2,\dots\}$ and also of the 
sets $\mathbb N\times\{-\}$ and $\{\dots,-d-2,-d-1\}$.
 The following behavior of the sequence $(l^n(\cdot)(R^n(\cdot)+d))$
is decisive for out argument.
Assume that for some $k$ the sequence $R^n$
overjumps the set $\{-2d,\dots,0\}$ from above:
\ba
R^n(k)>0 \quad \text{and}\quad R^n(k+1)<-2d.
\ea
Then by definition of transition probabilities \eqref{e:over} we have
\ba
l^n(k) (R^n(k)+d)=l^n(k+1)(R^n(k+1)+d),
\ea
i.e., in this case the sequence $(l^n(\cdot)(R^n(\cdot)+d))$ spends two time steps at the same point.
Let us construct a new sequence $(\widehat X^n(k))_{k\geq 0}$ from the sequence $(l^n(k)(R^n(k)+d))_{k\geq 0}$ that skips one of these steps.

We define a random time change as follows.
Set
\ba
&\widehat \lambda_n(0):=0,\\
&\widehat \lambda_n(k):= \sum_{j=1}^k \Big(1- \1(R^n(k+1)>0,\ R^n(k)< -2d)\Big),\\
&\widehat \lambda_n^{-1}(k):=\inf\{j\geq 0\colon  \widehat \lambda_n(j)=k\},\\
&\widehat X^n(k):= l^n(\widehat\lambda_n^{-1}(k))(R^n(\widehat\lambda_n^{-1}(k))+d).
\ea
Notice that $\widehat X^n$ is a Markov chain with transition probabilities
\ba
\label{e:622}
\Pb(\widehat X^n(1)=j\ | \ \widehat X^n(0)=i)=
\begin{cases}
    \Pb(X(1)=j\ | \ X(0)=i), &  |i|>d,\\
    \Pb(X(\hat \tau_0)=j\ | \ X(0)=i), &  |i|\leq d,
\end{cases}
\ea
where $(X(k))_{k\geq 0}$ is the original perturbed random walk and  
$\widehat \tau_0:=\inf\{k\geq 0\colon |X(k)|> d\}$, see also  \eqref{eq:defn_of_exits_entra}.

Let us show that the number of skipped steps is negligible as $n\to\infty$ and convergence
\ba
\label{e:convXhat}
\vf\Big(\frac{\widehat X^n([n\cdot])}{\sqrt{n}}\Big) \Rightarrow W^\text{skew}_\gamma(\cdot, x),\quad n\to\infty.
\ea
holds true.

Indeed, we have the inequality
\ba
\label{e:6a}
{}[nt]-\widehat \lambda_n([nt])\leq \sqrt{n}L^n(t) \max(v_-, v_+), \quad t>0,
\ea
where $L^n $ is defined in \eqref{eq:L338}.
It follows from the weak convergence of $(L^n(\cdot))$ established Theorem \ref{thm:Walsh_lim_pertRW}, 
that for all $t>0$
\ba
\label{e:6b}
 \Big| \frac{\widehat \lambda_n([nt])}{n}-t\Big|
\leq   \frac{L^n(t) \max(v_-, v_+)}{\sqrt{n}} +\frac1n
 \overset{\Pb}{\to} 0,\ n\to\infty.
\ea
Hence for any $T>0$ we have uniform convergence
\ba
\label{e:6c}
\sup_{t\in[0,T]}\Big| \frac{\widehat \lambda_n([nt])}{n}-t\Big|\overset{\Pb}{\to} 0,\ n\to\infty,
\ea
and by monotonicity of $t\mapsto \widehat \lambda_n([nt])$ 
\ba
\sup_{t\in[0,T]}\Big| \frac{\widehat \lambda_n^{-1}([nt])}{n}-t\Big|\overset{\Pb}{\to} 0,\ n\to\infty.
\ea
Therefore convergence \eqref{e:convXhat} follows by Corollary \ref{corl:composition ordinary}. 

\medskip
\noindent
\textbf{Step 3.}
Now we consider the 
original Markov chain $(X^n(k))_{k\geq 0}$ 
and define a random time change that removes times that $X^n$ spends in the set $\{-d,\dots,d\}$ after entering this set. Let us introduce a sequence
$\{\lambda_n (k)\}_{k\geq 0}$ as follows
\ba
\lambda_n(0)&:=0,\\
\lambda_n(k)&:= \sum_{j=1}^k \Big(\1_{|X^n(j)|>d}+  \1_{|X^n(j)|\leq d,\ |X^n(j-1)|>d}\Big)\\
&=k-  \sum_{j=1}^k  \1_{|X^n(j-1)|\leq d, |X^n(j)|\leq d},\quad k\geq 1,
\ea
The sequence $\lambda_n$ is a time homogeneous additive functional of a Markov chain $X^n$
that counts all the steps
of $X^n$ except those that start and end within the membrane. The inverse mapping
\ba
\lambda_n^{-1}(k) := \inf\{j \geq  0\colon \lambda_n(j) \geq  k\},\quad  k\geq  0,
\ea
is a sequence of stopping times, so that the process $k\mapsto X^n(\lambda^{-1}_n(k))$
is a time homogeneous Markov chain.
Assume that $X^n (0) = i\in\mathbb Z$. Then
\ba
\lambda_n^{-1}(1)=
\begin{cases}
1, \quad \text{if } |i| > d,\\
\widehat\tau_0 = \inf\{k\geq  0 \colon |X(k)| > d\}, \quad\text{if } |i| \leq d.
\end{cases}
\ea
Therefore, \eqref{e:622} yields
\ba
\Big( \widehat X^n(k)\Big)_{k\geq 0} \overset{\di}{=}\Big(X^n (\lambda_n^{-1}(k)) \Big)_{k\geq 0}.
\ea
From now on, without loss of generality
we assume that 
\ba
\Big(\frac{\widehat X^n([nt])}{\sqrt n}\Big)_{t\geq 0} = \Big(\frac{X^n (\lambda_n^{-1}([nt]))}{\sqrt n}\Big)_{t\geq 0}.
\ea

\medskip
\noindent
\textbf{Step 4.}
In order prove the weak convergence we will apply Corollary \ref{corl:randomWalkTimeChangeConvergence} and use the following representation:
\ba
\Big(\frac{X^n (\lambda_n^{-1}([nt]))}{\sqrt n}\Big)_{t\geq0}
=\Big(\frac{X^n (n\cdot ) }{\sqrt n}\circ {\eta_n(t)}\Big)_{t\geq0},
\ea
where $\eta_n(t)=\frac{\lambda_n^{-1}([nt])}{n}$. 
 Convergence \eqref{corl:randomWalkTimeChangeConvergence:condition} in Corollary \ref{corl:randomWalkTimeChangeConvergence} holds true because 
the supremum corresponding to \eqref{corl:randomWalkTimeChangeConvergence:condition}
does not exceed $\frac{2d}{\sqrt n}$. Hence it is left to check the condition \eqref{corl:randomWalkTimeChangeConvergence:lambda}.

Let 
\ba
\widehat N^n(k)=\sum_{j=0}^k \1_{|\widehat X^n(j)|\leq d}
\ea
be the number of visits of $\widehat X^n$ to the membrane, i.e., to the set $\{-d,\dots,d\}$.
We define 
\ba 
&\widehat \sigma_0^n:=0,\\
&\widehat \tau_k^n:=\inf\{j\geq\widehat \sigma_k^n\colon \ |X^n(j)|>d \}, \\
&\widehat \sigma_{k+1}^n=\inf\{j>\widehat \tau_k^n\colon  |X^n(j)|\leq  d\},\quad k\geq 0.
\ea
Denote the time spent in the membrane by
\ba
\zeta^n_k:=\widehat\tau_{k}^n-\widehat\sigma_{k}^n.
\ea
Due to Assumption \textbf{B}$_2$, the random variables $\zeta^n_k$ have finite moments of all orders.
We note that
\ba
k-\lambda_n(k)\leq \sum_{j=0}^{ \widehat N^n(k) }\zeta^n_j,
\ea
so that
\ba
\Big|t-\frac{\lambda_n([nt])}{n}\Big|\leq  \frac{1}{n} + \frac{1}{n}\sum_{j=0}^{ \widehat N^n([nt]) }\zeta^n_j.
\ea
Note that 
\ba
\frac{\widehat N^n([nt])}{\sqrt n}\leq  L^n(t)\max(v_-,v_+),
\ea
where $L^n$ is defined in \eqref{eq:L338} (compare with \eqref{e:6a}).
Similarly to \eqref{e:6b}, for any $t\geq 0$ we have
\ba
\label{e:NN}
\frac{\widehat N^n([nt])}{n}\stackrel{\P}{\to} 0,\quad n\to \infty.
\ea
For $|x|\leq d$, denote 
\ba
f(x):=\E[\zeta^n_0|X^n(0)=x]
\ea
and note that $f(\cdot)$ does not depend on $n$ and is bounded.
We write:
\ba
\sum_{j=0}^{ \widehat N^n([nt]) }\zeta^n_j = \sum_{j=0}^{ \widehat N^n([nt]) }\Big(\zeta^n_j -f(X^n(\widehat\sigma^n_j))\Big) 
+\sum_{j=0}^{ \widehat N^n([nt]) } f(X^n(\widehat\sigma^n_j)) .
\ea
Therefore we get
\ba
\label{e:fN}
\frac{1}{n}\sum_{j=0}^{ \widehat N^n([nt]) } f(X^n(\widehat\sigma^n_j)) \stackrel{\P}{\to} 0,\quad n\to \infty.
\ea
The random sequence
\ba
k\mapsto \sum_{j=0}^{k}\Big(\zeta^n_j -f(X^n(\widehat\sigma^n_j))\Big) 
\ea
is a martingale, so that by Doob's inequality for any $N\geq 1$
\ba
\label{e:doob}
\E \max_{0\leq k\leq N } \Big|\sum_{j=0}^{k}\Big(\zeta^n_j -f(X^n(\widehat\sigma^n_j))\Big)\Big|^2 
&\leq 4 \E \sum_{j=0}^{N}\Big(\zeta^n_j -f(X^n(\widehat\sigma^n_j))\Big)^2\\ 
&\leq 4(N+1)\max_{|x|\leq d} \E\Big[ \Big(\zeta^n_0 -f(X^n(\widehat\sigma^n_0))\Big)^2  \Big|X^n(0)=x\Big]\\
&\leq C(N+1),
\ea
for some $C>0$. Hence for any $\delta>0$
\ba
\frac1n\Big|
\sum_{j=0}^{ \widehat N^n([nt]) }\Big(\zeta^n_j -f(X^n(\widehat\sigma^n_j))\Big) \Big|
&\leq \frac1n\1(\widehat N^n([nt])\leq n\delta )\Big|
\sum_{j=0}^{ \widehat N^n([nt]) }\Big(\zeta^n_j -f(X^n(\widehat\sigma^n_j))\Big) \Big|\\
&+\frac1n\1(\widehat N^n([nt])> n\delta )\Big|
\sum_{j=0}^{ \widehat N^n([nt]) }\Big(\zeta^n_j -f(X^n(\widehat\sigma^n_j))\Big) \Big|\\
&\leq \frac1n\max_{0\leq k\leq n\delta}\Big|
\sum_{j=0}^{ k}\Big(\zeta^n_j -f(X^n(\widehat\sigma^n_j))\Big) \Big|\\
&+\frac1n\1\Big(\frac{\widehat N^n([nt])}{n}> \delta\Big)\Big|
\sum_{j=0}^{ \widehat N^n([nt]) }\Big(\zeta^n_j -f(X^n(\widehat\sigma^n_j))\Big) \Big|\\
\ea
The first term in the latter formula converges to zero in mean square sense due to \eqref{e:doob} whereas
the second term converges to zero in probability due to \eqref{e:NN}.
Combining all the above estimates yields
\ba
\Big|\frac{\lambda_n([nt])}{n}-t\Big| \stackrel{\P}{\to} 0,\quad n\to \infty.
\ea
Due to monotonicity, for any $T>0$ we have convergence:
\ba
&\sup_{t\in[0,T]}\Big| \frac{\lambda_n([nt])}{n}-t\Big|\overset{\Pb}{\to} 0,\ n\to\infty,\\
&\sup_{t\in[0,T]}\Big| \frac{\lambda_n^{-1}([nt])}{n}-t\Big|\overset{\Pb}{\to} 0,\ n\to\infty.
\ea
Hence, the application of Corollary \ref{corl:randomWalkTimeChangeConvergence} implies the desired convergence 
\ba
\frac{X^n([n\cdot])}{\sqrt{n}} \Rightarrow W^\text{skew}_\gamma(\cdot, x),\quad n\to\infty.
\ea
Inspecting the construction steps of the Markov chains $\cX^n$, $\widehat X^n$ and $X^n$ we conclude that
the parameter $\gamma$ defined in \eqref{e:gamma1} coincides with $\gamma$ in \eqref{eq:gamma_probab}.
\end{proof}

\section{Proofs of 
Corollary \ref{c:sobm} and Theorem \ref{thm:Walsh_non_rigorous}\label{section:proof_corl_1_2}}


\begin{proof}[Proof of Corollary \ref{c:sobm}]
 
Denote for brevity the inverse function
\begin{equation}
\psi(x):=\varphi^{-1}(x)=x( v_+ \1_{x\geq 0} +v_- \1_{x< 0})
\end{equation}
Recall the SDE \eqref{e:SDEskew} and apply the It\^o-Tanaka formula with symmetric semimartingale local time to $\psi(W^\text{skew}_\gamma)$, see
\cite[Exercise 1.25, Chapter VI]{RevuzYor05}:
\begin{equation}
\label{e:XWgamma}
\begin{aligned}
Y(t)=\psi(W^\text{skew}_\gamma(t))&:=x
+ \int_0^t \Big(v_+\1_{W^\text{skew}_\gamma(s)>0}+ v_-\1_{W^\text{skew}_\gamma(s)<0}\Big)\,\di W(s)\\
&+\frac12 \gamma (v_+ +v_-)L_0^{W^\text{skew}_\gamma}(t) + \frac12 (v_+-v_-)L_0^{W^\text{skew}_\gamma}(t) \\
=&x
+ \int_0^t \Big(v_+\1_{Y(s)>0}+ v_-\1_{Y(s)<0})\Big)\,\di W(s) \\
&+\frac12 \gamma (v_+ +v_-)L_0^{W^\text{skew}_\gamma}(t) + \frac12 (v_+-v_-)L_0^{W^\text{skew}_\gamma}(t)
\end{aligned}
\end{equation}
and 
\begin{equation}
\label{e:Y1}
\begin{aligned}
|Y(t)|=|\psi(W^\text{skew}_\gamma(t))|&:=|x|
+ \int_0^t \Big(v_+\1_{W^\text{skew}_\gamma(s)>0}-v_-\1_{W^\text{skew}_\gamma(s)<0}\Big)\,\di W(s)\\
&+\frac12 \gamma (v_+ -v_-)L_0^{W^\text{skew}_\gamma}(t)+ \frac12 (v_+ +v_-)L_0^{W^\text{skew}_\gamma}(t).
\end{aligned}
\end{equation}
On the other hand, 
the process $Y$ is a continuous semimartingale itself. Hence, the It\^o-Tanaka formula applied to the process $Y$ yields
\begin{equation}
\label{e:Y2}
\begin{aligned}
|Y(t)|&=|x| 
+ \int_0^t \Big(\1_{Y(s)>0}- \1_{Y(s)<0}\Big)\,\di Y(s)+
  L_0^Y(t)\\
&=\int_0^t \Big(v_+\1_{W^\text{skew}_\gamma(s)>0}-v_-\1_{W^\text{skew}_\gamma(s)<0}\Big)\,\di W(s)+
 L_0^Y(t),
\end{aligned}
\end{equation}
where $L_0^Y(t)$ is the symmetric semimartingale local time defined by
\ba
L_0^Y(t)=\lim_{\ve\to0+}\frac{1}{2\ve}\int_0^t \1_{|Y(s)|<\ve} \di \langle Y, Y\rangle_s= \lim_{\ve\to0+}\frac{1}{2\ve}\int_0^t \Big(v_+^2\1_{Y(s)>0}+v_-^2\1_{Y(s)<0}\Big)\,\di s.
\ea
Comparing the representations \eqref{e:Y1} and \eqref{e:Y2} we conclude that
\begin{equation}
\begin{aligned}
L_0^Y(t)=\frac12\Big[ \gamma (v_+ -v_-)+ (v_+ +v_-)\Big]L_0^{W^\text{skew}_\gamma}(t).
\end{aligned}
\end{equation}
Substituting this formula into \eqref{e:XWgamma} and recalling  that $Y$ spends zero time at zero  we get 
\eqref{e:Xinf}.
\end{proof}

\begin{proof}[Proof of Theorem \ref{thm:Walsh_non_rigorous}]
Let $\cX$ be a Markov chain on $\cN^m\cup\{0\}$ from Theorem~\ref{thm:Walsh_non_rigorous}. Since the processes from Theorems 
\ref{thm:Walsh_non_rigorous}
and \ref{thm:Walsh_lim_pertRW} live on different state spaces, first we construct a Markov chain $\widehat\cX$ on $\cZ^m$ related to $\cX$
that satisfy assumptions of Theorem \ref{thm:Walsh_lim_pertRW}. 

In assumption $\mathbf{A}_1$, let the distribution of $\xi^{(i)}$ coincide with the distribution of $\xi^{(i)}_k$, $i=1,\dots,m$.
In particular, $\widehat\cX$ may visit the set $\{\dots,-1,0\}\times \{1,\dots,m\}$.

The jumps from the state $(0,i)$ are defined as follows:
\ba
\P\big(\widehat \cX(1)=( \widehat y,j)\,|\, \widehat \cX(0)=(0,i)  \big)
:=\P\big(\cX(1)=(\widehat y,j)\,|\, \cX(0)=0  \big),\quad i,j=1,\dots,m.
\ea
For $\widehat x<0$ and $\widehat y\in\mbN$ we set
\ba
\P\big(\widehat \cX(1)=(\widehat y,j)\,|\, \widehat \cX(0)=(\widehat x,i)  \big)
:= \P(\cX(1)=(\widehat y,j)\,|\, R(0)+\xi^{(i)}_1=\widehat x,l(0)=i \big),
\ea
where $\cX(0)=(R(0),l(0))$.

Let the initial distributions of $\widehat \cX(0)$ and $\cX(0)$ coincide in the following sense:
\ba
&\P\big(\widehat \cX(0)=(\widehat y,j) \big)
:=\P(\cX(0)=(\widehat y,j)\big),\quad \widehat y\in\mbN,\ j=1,\dots,m,  \\
&\P\big(\widehat \cX(0)=(0,1) \big):=\P(\cX(0)=0\big),
\ea
and $\P(\widehat \cX(0)=(\widehat y,j))=0$ in all other cases.

According to
Theorem \ref{thm:Walsh_lim_pertRW}, the scaling limit of $\widehat \cX$ is a WBM starting at 0.

Let 
\ba
\widehat \lambda(k)=\sum_{i=1}^{k} \1( \widehat R(i)\geq 0 )  .
\ea
The random time change $\widehat \lambda^{-1}(\cdot)$ removes time instants when $\widehat \cX$ visits the set  $\{\dots,-1,0\}\times \{1,\dots,m\}$.
It follows from the construction of $\widehat\cX$ and $\widehat\lambda$ that
\ba
\Big(\widehat  R(\widehat  \lambda^{-1}(k))\1_{l(\widehat  \lambda^{-1}(k))=i},\ i=1,\dots,m \Big)_{k\geq 0} 
\stackrel{\di}{=} \Big(R(k)\1_{l(k)=i},\ i=1,\dots,m \Big)_{k\geq 0} .
\ea
Following the argument of Step 2 in the proof of Theorem \ref{thm:Skew_lim_pertRW} we see that  scaling limit of the original Markov chain 
$\cX$ is the same WBM as for the scaling limit of $\widehat \cX$.
\end{proof}

\bibliographystyle{plain}
\bibliography{extracted}

\end{document}